\documentclass[11pt]{amsart}

\usepackage{amssymb}
\usepackage{thmtools, thm-restate}
\newif\ifwantpicture 
\wantpicturetrue 

\usepackage{hyperref}
\usepackage{float}
\usepackage{mathtools} 

\usepackage[capitalise]{cleveref}

\usepackage{tikz, pgfplots,graphicx}
\usetikzlibrary{shapes,arrows,positioning,decorations.markings}
\usepackage{tkz-euclide}
\pgfplotsset{compat=1.11}

\usetikzlibrary{math}

\crefname{subsection}{subsection}{subsections}
\Crefname{subsection}{Subsection}{Subsections}

\newtheorem{thm}{Theorem}[section]

\crefformat{thm}{Theorem~#2#1#3}
\Crefformat{thm}{Theorem~#2#1#3}
\Crefname{thm}{Theorem}{Theorems}
\crefname{thm}{theorem}{theorems}
\crefmultiformat{thm}{Theorems~#2#1#3}{ and~#2#1#3}{, #2#1#3}{ and~#2#1#3}

\newtheorem{cor}[thm]{Corollary}

\newtheorem{lem}[thm]{Lemma}
\newtheorem{prop}[thm]{Proposition}

\crefformat{cor}{Corollary~#2#1#3}
\Crefformat{cor}{Corollary~#2#1#3}
\crefformat{conj}{Conjecture~#2#1#3}
\Crefformat{conj}{Conjecture~#2#1#3}
\crefformat{lem}{Lemma~#2#1#3}
\Crefformat{lem}{Lemma~#2#1#3}
\crefformat{prop}{Proposition~#2#1#3}
\Crefformat{prop}{Proposition~#2#1#3}

\Crefname{cor}{Corollary}{Corollaries}
\crefname{cor}{corollary}{corollaries}
\Crefname{lem}{Lemma}{Lemmas}
\crefname{lem}{lemma}{lemmas}
\Crefname{prop}{Proposition}{Propositions}
\crefname{prop}{proposition}{propositions}

\theoremstyle{definition}
\newtheorem{abuse}[thm]{Abuse of terminology}
\newtheorem*{acknowledgments}{Acknowledgments}
\newtheorem{assumption}[thm]{Assumption}
\newtheorem{definition}[thm]{Definition}
\newtheorem{definitions}[thm]{Definitions}
\newtheorem{example}[thm]{Example}
\newtheorem{notation}[thm]{Notation}
\newtheorem{problem}[thm]{Problem}
\newtheorem{observation}[thm]{Observation}
\newtheorem{question}[thm]{Question}

\newtheorem{remark}[thm]{Remark}

\newcounter{case}
\newenvironment{case}[1][\unskip]{\refstepcounter{case}\bf
\medskip \noindent Case \thecase\ #1. \it}{\unskip\upshape}
\crefformat{case}{Case~#2#1#3}
\Crefname{case}{Case}{Cases}
\crefname{case}{case}{cases}
\newcounter{caseholder}
\numberwithin{case}{caseholder}
\renewcommand{\thecase}{\arabic{case}}

\newcounter{step}

\crefformat{step}{Step~#2#1#3}
\Crefname{atep}{Step}{Steps}
\crefname{step}{step}{steps}

\numberwithin{step}{stepholder}
\renewcommand{\thestep}{\arabic{step}}

\newcommand{\pref}[1]{(\ref{#1})}
\newcommand{\fullref}[2]{\ref{#1}\pref{#1-#2}}
\newcommand{\fullcref}[2]{\cref{#1}\pref{#1-#2}}
\newcommand{\fullCref}[2]{\Cref{#1}\pref{#1-#2}}
\newcommand{\Z}{\mathbb{Z}}  
\newcommand{\N}{\mathbb{N}}
\newcommand{\XX}{\mathsf{X}}
\renewcommand{\aa}{\mathsf{a}}

\newcommand{\aut}{\mathsf{Aut}(X)}
\newcommand{\autt}{\mathop{\mathsf{Aut}}}
\newcommand{\bdry}{\delta}

\newcommand{\edge}{\mathord{\vrule width 12pt height 3pt depth -1.5pt}}
\newcommand{\edgecut}{\mathcal{E}}
\newcommand{\lsim}[1]{\stackrel{#1}{\sim}}
\newcommand{\requiv}[1]{\stackrel{#1}{\equiv}}

\newcommand{\from}{\leftarrow} 

\newcommand{\cay}{\mathop{\mathsf{Cay}}}
\numberwithin{equation}{thm}

\makeatletter
\newcommand{\noprelistbreak}{\smallskip\@nobreaktrue\nopagebreak} 
\makeatother

\begin{document}

\title
[Vertex-transitive graphs with a unique hamiltonian cycle]%
{On vertex-transitive graphs with 
\\ a unique hamiltonian cycle}

\author{Babak Miraftab}
\address{Department of Mathematics and Computer Science, University of Lethbridge, 4401 University
Drive, Lethbridge, Alberta, T1K 3M4, Canada}
\curraddr{Computational Geometry Lab, School of Computer Science, Carleton University
Ottawa, Ontario, K1S 5B6 Canada}
\email{bobby.miraftab@gmail.com}

\author{Dave Witte Morris}
\address{Department of Mathematics and Computer Science, University of Lethbridge, 4401 University
Drive, Lethbridge, Alberta, T1K 3M4, Canada}
\email{dmorris@deductivepress.ca, https://deductivepress.ca/dmorris}

\begin{abstract}
A graph is said to be \emph{uniquely hamiltonian} if it has a unique hamiltonian cycle. For a natural extension of this concept to infinite graphs, we find all uniquely hamiltonian vertex-transitive graphs with finitely many ends, and also discuss some examples with infinitely many ends. In particular, we show each nonabelian free group~$F_n$ has a Cayley graph of degree $2n + 2$ that has a unique hamiltonian circle. (A weaker statement had been conjectured by A.\,Georgakopoulos.) Furthermore, we prove that these Cayley graphs of~$F_n$ are outerplanar.
\end{abstract}

\maketitle


\section{Introduction}

A graph is said to be \emph{uniquely hamiltonian} if it has a unique hamiltonian cycle. The first major result on this topic is the following well-known \lcnamecref{smith}, which implies that cubic graphs are never uniquely hamiltonian.

\begin{thm}[C.\,A.\,B.\,Smith, 1946 {\cite{MR19300}}]\label{smith}
Every edge of any finite, cubic graph is contained in an even number of hamiltonian cycles.
\end{thm}

It was conjectured by J.\,Sheehan \cite{Sheehan} in 1974 that regular graphs of degree~$4$ also are never uniquely hamiltonian. It is well known that, if true, this would imply the more general conjecture that no regular graphs (other than cycles) are uniquely hamiltonian (see \cite[p.~43]{MR3372339}, for example).
At present, it is known that an $r$-regular graph cannot be uniquely hamiltonian if either $r$ is odd \cite{MR499124} or $r > 22$ \cite{MR2290229}.
Vertex-transitivity (see \cref{VertTransDefn}) is a stronger condition than regularity, and the conjecture has been established for these much more special graphs, without any restriction on the degree (see the start of 
\cref{finmanyends} for a short proof):

\begin{restatable}[M.\,\v{S}ajna and A.\,Wagner {\cite[Cor.~2.2]{MR3372339}}]{prop}{FiniteCase} 
  \label{FiniteCase}
Cycles are the only finite, vertex-transitive graphs with a unique hamiltonian cycle.
\end{restatable}
This paper addresses the extension of the above \lcnamecref{FiniteCase} to infinite graphs. This is a special case of a question of B.\,Mohar \cite{mohar}, which asks to what extent results on unique hamiltonicity in finite graphs generalize to the infinite setting.
(See \cite{Heuer} and~\cite{max} for some important progress on this problem.)

The first step is to choose an analogue of the notion of ``hamiltonian cycle'' that applies to infinite graphs. Classically, the following notion was used:

\begin{definition}[cf.\ {\cite[pp.~286 and 297]{NashWilliams-Infinite}}] \label{2wayHPDefn}
A \emph{two-way-infinite hamiltonian path} in a countably infinite graph~$X$ is a two-way-infinite list 
    \[ \ldots, x_{-2}, x_{-1}, x_{0}, x_{1}, x_{2}, \ldots\]
of all the vertices of~$X$ (without any repeats), such that $x_i$ is adjacent to~$x_{i+1}$ for each~$i$.
\end{definition}

In this setting, \cref{FiniteCase} generalizes in the obvious way:

\begin{restatable}{prop}{wayUHP}\label{2wayUHP}
The two-way-infinite path~$P_\infty$ is the only vertex-transitive graph with a unique two-way-infinite hamiltonian path.
    \ifwantpicture
    \[ \begin{tikzpicture}
        \draw[line width=0.5mm] (-2.5,0)--(3.5,0);
	    \foreach \i in {-2,...,3}{
	       \draw (\i,0) node [circle,fill, inner sep=2pt] {};}
    \end{tikzpicture} \]
    \fi 
\end{restatable}

More recently, R.\,Diestel and others (cf.~\cite{Diestel-CycleSpace}) introduced a topological point of view that suggests a different analogue of hamiltonian cycle for infinite graphs (that are locally finite). The inspiration is that:
\begin{itemize}
    \item a hamiltonian cycle in a finite graph is a copy of the circle~$S^1$ in~$X$ (that contains all of the vertices of~$X$),
    and 
    \item adding the point at infinity to a two-way-infinite hamiltonian path in an infinite graph~$X$ yields a copy of~$S^1$ in the one-point compactification of~$X$ (that contains all of the vertices of~$X$).
\end{itemize}
The new approach recognizes that circles in a different compactification (specifically, the Freudenthal compactification) can be more interesting.

\begin{definitions}[\cite{Diestel-CycleSpace}]
Let $X$ be a connected, locally finite graph. 
    \begin{enumerate}
     
    \item The \emph{Freudenthal compactification} of~$X$ (denoted~$|X|$) is a topological space that is obtained from~$X$ (considered as a $1$-dimensional simplicial complex) by adding a point at infinity for each ``end'' of~$X$. (An intuitive understanding suffices for our purposes, but the interested reader can find a rigorous definition in \cite[\S4]{Diestel-CycleSpace}.) For example, the graphs pictured in \cref{2endedUHC} below each have exactly two ends (one end is to the left, and the other is to the right), because the \emph{number of ends} of~$X$ is the smallest $k \in \N \cup \infty$, such that, for every finite subset~$K$ of $V(X)$, the vertex-deleted graph $X \smallsetminus K$ has at most~$k$ infinite connected components. 
    
    \item A \emph{hamiltonian circle} in~$X$ is an embedded copy of the circle~$S^1$ in~$|X|$ that contains all of the vertices of~$X$  \cite[p.~75]{Diestel-CycleSpace}.
    \end{enumerate}
\end{definitions}

\begin{abuse}
By definition, if $C$ is a hamiltonian circle in a graph~$X$, then $C$ is a subset of the Freudenthal compactification~$|X|$. Its intersection with~$X$ is a spanning $2$-regular subgraph (i.e., a \emph{$2$-factor}) of~$X$, and it is traditional to abuse terminology by saying that this $2$-factor is a hamiltonian circle (even though, technically speaking, the hamiltonian circle is the union of the $2$-factor with the set of all ends of~$X$, or, equivalently, the hamiltonian circle is the \emph{closure} of this $2$-factor in the topological space~$|X|$).
\end{abuse}

If an infinite graph has only one end, then its Freudenthal compactification is the same as its one-point compactification, so it is easy to see that every hamiltonian circle is a two-way-infinite hamiltonian path. Thus, the following observation is immediate from \cref{2wayUHP} (since the two-way-infinite path~$P_\infty$ obviously does not have only one end):

\begin{cor}\label{nooneend}
No one-ended, vertex-transitive graph has a unique hamiltonian circle.
\end{cor}

If a graph~$X$ has two ends, then it is compactified by adding two points at infinity. In this case, we have the following description of hamiltonian circles in~$X$:

\begin{lem}[{\cite[Theorem 2.5]{MR2800969}}]\label{hamin2ended}
Let $X$ be a graph whose Freudenthal compactification has exactly two points at infinity; call them $+\infty$ and~$-\infty$.
A $2$-factor~$C$ of~$X$ is a hamiltonian circle if and only if it is the vertex-disjoint union of two two-way-infinite paths $P_1$ and $P_2$, each of which has one end at~$-\infty$ and the other end at~$+\infty$.

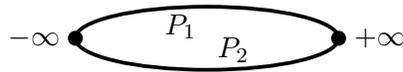
\[
\begin{tikzpicture}
        \begin{scope}[scale=0.7]
        \draw[line width=0.5mm] (0,0) arc(180:0:2.5cm and 0.6cm);
	\draw (2,0.65) node[below] {$P_1$};
        \draw[line width=0.5mm] (0,0) arc(180:0:2.5cm and -0.6cm);
	\draw (3,-0.65) node[above] {$P_2$};
	\draw (0,0) node [circle,fill, inner sep=2pt] {};
	\draw (-0.1,0) node [left] {$-\infty$};
	\draw (5,0) node [circle,fill, inner sep=2pt] {};
	\draw (5.1,0) node [right] {$+\infty$};
        \end{scope}
        \end{tikzpicture} 
\]  
\end{lem}

From this description, it is clear that the two graphs in the following result do each have a hamiltonian circle. (In each picture, the edges of the hamiltonian circle are black, and the other edges of the graph are white.)

\begin{restatable}{thm}{secondmain}\label{2endedUHC}
Up to isomorphism, the only two-ended, vertex-transitive graphs with a unique hamiltonian circle are:
	\begin{enumerate}

	\item \label{2endedUHC-path}
	the square of a two-way-infinite path  
        \ifwantpicture
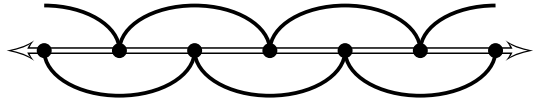
        \[
		\begin{tikzpicture}
	\begin{scope}[decoration={markings,mark=at position 
        0.59 with {\arrow{>}}}] 
	\draw[<->,>=latex',line width=0.9mm] (-3.5,1)--(3.5,1);
	\draw[<->,>=latex',line width=0.5mm, white] (-3.43,1)--(3.43,1);
        \draw[line width=0.5mm] (-2,1) arc(0:90:1cm and 0.6cm);
        \draw[line width=0.5mm] (-2,1) arc(180:0:1cm and 0.6cm);
        \draw[line width=0.5mm,] (0,1) arc(180:0:1cm and 0.6cm);
        \draw[line width=0.5mm,] (2,1) arc(180:90:1cm and 0.6cm);
        \draw[line width=0.5mm,] (-1,1) arc(180:360:1cm and 0.6cm);
        \draw[line width=0.5mm,] (1,1) arc(180:360:1cm and 0.6cm);
        \draw[line width=0.5mm,] (-3,1) arc(180:360:1cm and 0.6cm);
	\end{scope}
	\foreach \i in {-3,...,3}{
	\draw (\i,1) node [circle,fill, inner sep=2pt] {};}
        \end{tikzpicture}
        \]
        \fi 

	\item \label{2endedUHC-ladder}
	the two-way-infinite ladder	
        \ifwantpicture
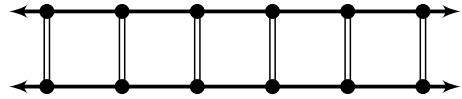
        \[ \begin{tikzpicture}
	\begin{scope}[decoration={markings,mark=at position 
        0.59 with {\arrow{>}}}] 
	\draw[->,>=latex',line width=0.5mm] (3,1)--(3.5,1);
	\draw[->,>=latex', line width=0.5mm] (3,0)--(3.5,0);
	\draw[->,>=latex',,line width=0.5mm] (-2,1)-- 
        (-2.5,1);
	\draw[->,>=latex',, line width=0.5mm] (-2,0)--     
        (-2.5,0);
	\end{scope}
	\foreach \i in {-2,...,3}{
		\draw[line width=0.9mm] (\i,1)--(\i,0);
		\draw[line width=0.5mm, white] (\i,1)--(\i,0);
		}
	\foreach \i in {-2,0,2}{
	\draw[line width=0.5mm] (\i,1)--(\i+1,1);}
	\foreach \i in {-2,0,2}{
	\draw[line width=0.5mm] (\i,0)--(\i+1,0);}
	\foreach \i in {-1,1}{
	\draw[,line width=0.5mm] (\i,0)--(\i+1,0);}
	\foreach \i in {-1,1}{
	\draw[,line width=0.5mm] (\i,1)--(\i+1,1);}
	\foreach \i in {-2,...,3}{
	\draw (\i,1) node [circle,fill, inner sep=2pt] {};
	\draw (\i,0) node [circle,fill, inner sep=2pt] {};}
        \end{tikzpicture} \]
    \fi 
	\end{enumerate}
\end{restatable} 

It is well known \cite[Prop.~6.1]{Mohar-relations} that the number of ends of a (locally finite) vertex-transitive graph is either $0$, $1$, $2$, or~$\infty$ (and that a graph is finite if and only if it has $0$~ends). Therefore, the above results settle all cases where the number of ends is finite.

S.\,Legge found the first example of a vertex-transitive graph with infinitely many ends that has a unique hamiltonian circle. In fact, he discovered an infinite family of examples. We will now provide an informal, geometric description of these graphs, but a formal construction (and a rigorous proof that they are uniquely hamiltonian) can be found in \cref{LeggeSect}.

\begin{example}[S.\,Legge {[personal communication]}] \label{LeggeEg}
Fix $m \ge 3$, and start with the natural planar embedding of a regular tree~$T_m$ of degree~$m$. (This is also known as a \emph{Bethe lattice}.) Construct an associated graph~$X_m$ that is embedded in the plane, by replacing each vertex of~$T_m$ with an $m$-cycle, and replacing each edge of~$T_m$ with two parallel edges. (More precisely, replace each vertex~$v$ of~$T_m$ with an $m$-cycle~$C_v$ that has an edge facing each vertex adjacent to~$v$. Now, for each pair $v,w$ of adjacent vertices of~$T_m$, we add two edges to~$X_m$, so that the two edges of $C_v$ and~$C_w$ that face each other are contained in a $4$-cycle.) For example, \cref{T3X3} has pictures of the tree~$T_3$ and the corresponding graph~$X_3$.
It is clear that $X_m$ is vertex-transitive and has infinitely many ends (because it was constructed from~$T_m$, which is arc-transitive and has infinitely many ends.) 

    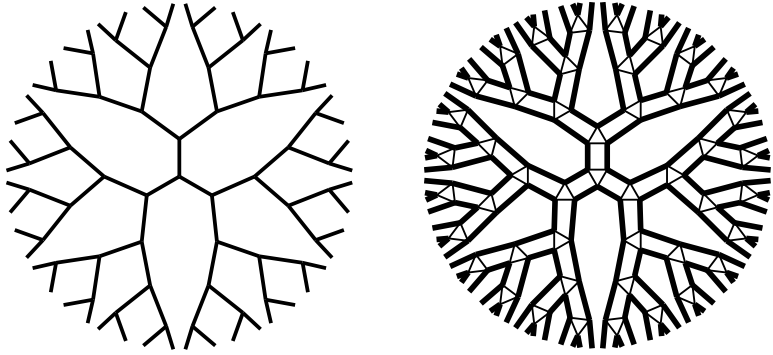
\begin{figure}[t]
        \begin{tikzpicture}[scale=0.5]
        \foreach \i in {0,1,2}{ 
        		\coordinate (A) at ({cos(\i*120 - 30)}, {sin(\i*120 - 30)});
        		\draw[line width=0.5mm]  (0,0)--(A);
		 \foreach \j in {-1,1}{ 
			\coordinate (B) at ({2*cos(\i*120 + \j*30 - 30)}, {2*sin(\i*120 + \j*30 - 30)});
			\draw[line width=0.5mm]  (A)--(B);
			\foreach \k in {-1,1}{ 
				\coordinate (C) at ({3*cos(\i*120 + \j*30 + \k*15 - 30)}, {3*sin(\i*120 + \j*30 + \k*15 - 30)});
				\draw[line width=0.5mm]  (B)--(C);
				\foreach \l in {-1,1}{ 
					\coordinate (D) at ({4*cos(\i*120 + \j*30 + \k*15 + \l*10 - 30)}, 
						{4*sin(\i*120 + \j*30 + \k*15 + \l*10 - 30)});
					\draw[line width=0.5mm]  (C)--(D);
					\foreach \m in {-1,1}{ 
						\coordinate (E) at ({4.6*cos(\i*120 + \j*30 + \k*15 + \l*8 + \m*5 - 30)}, 
							{4.6*sin(\i*120 + \j*30 + \k*15 + \l*8 + \m*5 - 30)});
						\draw[line width=0.5mm]  (D)--(E);
						}
					}
				}
		 	}
		}
        \end{tikzpicture}
        \qquad
         \begin{tikzpicture}[scale=0.5]
         \tikzmath{
        		function Ax(\x){
			return 0.3*cos(240 - 30);
			};
        		function Ay(\x){
			return 0.3*sin(240 - 30);
			};
        		function Bx(\x){
			return 0.3*cos(0 - 30);
			};
        		function By(\x){
			return 0.3*sin(0 - 30);
			};
        		function Cx(\x){
			return 0.3*cos(120-30);
			};
        		function Cy(\x){
			return 0.3*sin(120-30);
			};
		function Px(\r,\t) {
			return Ax() * cos(\t) + (Ay() + \r)*sin(\t);
			};
		function Py(\r,\t) {
			return -Ax() *sin(\t) + (Ay() + \r)*cos(\t);
			};
		function Qx(\r,\t) { 
			return Bx() * cos(\t) + (By() + \r)*sin(\t);
			};
		function Qy(\r,\t) {
			return -Bx()*sin(\t) + (By() + \r)*cos(\t);
			};
		function Rx(\r,\t) {
			return Cx() * cos(\t) + (Cy() + \r)*sin(\t));
			};
		function Ry(\r,\t) {
			return -Cx() *sin(\t) + (Cy() + \r)*cos(\t);
			};
	};

	\def\a{0}
	\def\b{0}
   	\draw[line width=0.25mm] ({-Px(\a,\b)},{-Py(\a,\b)})--({-Qx(\a,\b)},{-Qy(\a,\b)})--({-Rx(\a,\b)},{-Ry(\a,\b)})--cycle;
	 \foreach \i in {0,1,2}{ 
		\def\a{1}
		\def\b{\i*120}
   		\draw[line width=0.25mm] ({Px(\a,\b)},{Py(\a,\b)})--({Qx(\a,\b)},{Qy(\a,\b)})--({Rx(\a,\b)},{Ry(\a,\b)})--cycle;
   		\draw[line width=0.75mm] ({Px(\a,\b)},{Py(\a,\b)})--({-Qx(0,\b)},{-Qy(0,\b)});
   		\draw[line width=0.75mm] ({Qx(\a,\b)},{Qy(\a,\b)})--({-Px(0,\b)},{-Py(0,\b)});
		\foreach \j in {-1,1}{ 
			\def\c{2}
			\def\d{\i*120 + \j*30}
   			\draw[line width=0.25mm] ({Px(\c,\d)},{Py(\c,\d)})--({Qx(\c,\d)},{Qy(\c,\d)})--({Rx(\c,\d)},{Ry(\c,\d)})--cycle;
			\ifnum\j=1
				{\draw[line width=0.75mm] ({Qx(2,\d)},{Qy(2,\d)})--({Qx(1,\b)},{Qy(1,\b)});}
				{\draw[line width=0.75mm] ({Px(2,\d)},{Py(2,\d)})--({Rx(1,\b)},{Ry(1,\b)});}
			\else
				{\draw[line width=0.75mm] ({Qx(2,\d)},{Qy(2,\d)})--({Rx(1,\b)},{Ry(1,\b)});}
				{\draw[line width=0.75mm] ({Px(2,\d)},{Py(2,\d)})--({Px(1,\b)},{Py(1,\b)});}
			\fi;
			\foreach \k in {-1,1}{ 
				\def\e{3}
				\def\f{\i*120 + \j*30 + \k*15}
				\draw[line width=0.25mm] ({Px(\e,\f)},{Py(\e,\f)})--({Qx(\e,\f)},{Qy(\e,\f)})--({Rx(\e,\f)},{Ry(\e,\f)})--cycle;
				\ifnum\k=1
					{\draw[line width=0.75mm] ({Qx(3,\f)},{Qy(3,\f)})--({Qx(2,\d)},{Qy(2,\d)});}
					{\draw[line width=0.75mm] ({Px(3,\f)},{Py(3,\f)})--({Rx(2,\d)},{Ry(2,\d)});}
				\else
					{\draw[line width=0.75mm] ({Qx(3,\f)},{Qy(3,\f)})--({Rx(2,\d)},{Ry(2,\d)});}
					{\draw[line width=0.75mm] ({Px(3,\f)},{Py(3,\f)})--({Px(2,\d)},{Py(2,\d)});}
				\fi;
				\foreach \l in {-1,1}{ 
					\def\g{4}
					\def\h{\i*120 + \j*30 + \k*15 + \l*8}
					\draw[line width=0.25mm] ({Px(\g,\h)},{Py(\g,\h)})--({Qx(\g,\h)},{Qy(\g,\h)})--({Rx(\g,\h)},{Ry(\g,\h)})--cycle;
					\ifnum\l=1
						{\draw[line width=0.75mm] ({Qx(4,\h)},{Qy(4,\h)})--({Qx(3,\f)},{Qy(3,\f)});}
						{\draw[line width=0.75mm] ({Px(4,\h)},{Py(4,\h)})--({Rx(3,\f)},{Ry(3,\f)});}
					\else
						{\draw[line width=0.75mm] ({Qx(4,\h)},{Qy(4,\h)})--({Rx(3,\f)},{Ry(3,\f)});}
						{\draw[line width=0.75mm] ({Px(4,\h)},{Py(4,\h)})--({Px(3,\f)},{Py(3,\f)});}
					\fi;
					\foreach \m in {-1,1}{ 
						\def\p{5}
						\def\q{\i*120 + \j*30 + \k*15 + \l*8 + \m*2}
						\ifnum\m=1
							{\draw[line width=0.75mm] ({Qx(4.75,\q)},{Qy(4.75,\q)})--({Qx(4,\h)},{Qy(4,\h)});}
							{\draw[line width=0.75mm] ({Px(4.75,\q)},{Py(4.75,\q)})--({Rx(4,\h)},{Ry(4,\h)});}
						\else
							{\draw[line width=0.75mm] ({Qx(4.75,\q)},{Qy(4.75,\q)})--({Rx(4,\h)},{Ry(4,\h)});}
							{\draw[line width=0.75mm] ({Px(4.75,\q)},{Py(4.75,\q)})--({Px(4,\h)},{Py(4,\h)});}
						\fi;
						}
					}
				}
			}
		}
	     \end{tikzpicture} 
    \caption{The tree $T_3$ and Legge's associated graph~$X_3$. The dark edges in $X_3$ are the unique hamiltonian circle.}
        \label{T3X3}
    \end{figure} 

Removing the edges of the $m$-cycles results in a $2$-factor that is easily seen to be a hamiltonian circle in~$X_m$. 
Furthermore, this hamiltonian circle is unique. To see this, note that removing any pair of edges of~$X_m$ that both correspond to the same edge in~$T_m$ will disconnect the graph. Speaking informally, any hamiltonian circle needs to reach \emph{and return from} every end, so it must contain both of these edges (see \cref{2factor}). Hence, every hamiltonian circle contains (and is therefore equal to) the hamiltonian circle described above.
\end{example}

Recall that a finite graph is \emph{outerplanar} if it (is 2-connected and) has a planar embedding, such that all of the vertices lie on the boundary of a single face.
This has the following natural generalization to infinite graphs:

\begin{definition}[cf.\ {\cite[Thm.~1.8]{Heuer}}] \label{OuterPlanarDef}
A graph~$X$ is \emph{outerplanar} if there is an embedding of~$|X|$ into a closed disk, such that some hamiltonian circle of~$|X|$ is mapped onto the boundary of the disk.
\end{definition}

From the perspective of this paper, it is important to note that K.\,Heuer proved every outerplanar graph has a unique hamiltonian circle (see \fullcref{HeuerOuterplanar}{UHC}). The following fact is fairly obvious:

\begin{observation} \label{LeggeOuter}
Legge's examples are outerplanar.
\end{observation}

Our remaining results deal with a certain very natural class of vertex-transitive graphs with infinitely many ends: Cayley graphs $\cay(F_n; S)$ of a free group~$F_n$ (with $n \ge 2$). (See \cref{CayleyDefn} for the definition of \emph{Cayley graph}.) So we fix the following notation:

\begin{notation}
Let $A = \{a_1, a_2, \ldots, a_n\}$ be the standard generating set of the free group~$F_n$ \textup(with $n \ge 2$\textup).
\end{notation}

First, an elementary observation:

\begin{lem}[see \cref{UniqueInCayley}]
If\/ $\cay(F_n; S)$ has a unique hamiltonian circle~$C$, then $C = \cay(F_n; s^{\pm1})$, for some $s \in S$.
\end{lem}

Since $\cay(F_n; A^{\pm1})$ is a tree, it is easy to see (and well known) that it does not have a hamiltonian circle. So it is natural to ask whether adding a single additional element~$s$ can result in a hamiltonian circle:

\begin{problem}\label{QF_n}
Find all $s\in F_n$, such that $\cay(F_n; s^{\pm 1})$ is a hamiltonian circle in $\cay(F_n; A \cup \{s^{\pm 1}\})$.
\end{problem}

\begin{remark}[see \cref{ChangeS}]
It is not difficult to see that if $\cay(F_n; s^{\pm 1})$ is a hamiltonian circle in $\cay(F_n; A \cup \{s^{\pm 1}\})$, then it is also a hamiltonian circle in $\cay(F_n; S)$, for every symmetric generating set~$S$ of~$F_n$ that contains~$s$.
\end{remark}

We will prove that the solution set of \cref{QF_n} is not the empty set, and A.\,Georgakopoulos has shown that the hamiltonian circles we construct are unique (see \cref{uniqueFn}):

\begin{restatable}{thm}{thirdmain}\label{FreeUHCEg}
In each of the following cases, $\cay(F_n; s^{\pm1})$ is the unique hamiltonian circle in $\cay \bigl( F_n; A^{\pm1} \cup \{s^{\pm1}\} \bigr)$\textup:
    \begin{enumerate}
    \item \label{FreeUHCEg-squares}
    $s = a_1^2 \, a_2^2 \cdots a_n^2$,
    or
    \item \label{FreeUHCEg-commutators}
    $n$ is even and $s = [a_1, a_2] [a_3, a_4] \cdots [a_{n-1}, a_n]$.
    \end{enumerate}
\end{restatable} 

Legge's example \pref{LeggeEg} has the great virtue that the hamiltonian circle is easily visualized. Unfortunately, the hamiltonian circles in \cref{FreeUHCEg} seem to be much less intuitive. However, the following result suggests that it may be possible to draw reasonable pictures of them.

\begin{thm}[see \cref{MoreFnOuterplanar}] \label{FnOuterplanar}
$\cay \bigl( F_n; A^{\pm1} \cup \{s^{\pm1}\} \bigr)$ is outerplanar if $s$ and~$n$ are as in \fullref{FreeUHCEg}{squares} or \fullref{FreeUHCEg}{commutators}.
\end{thm}

\begin{question}
Is there a vertex-transitive graph that has a unique hamiltonian circle, but is not outerplanar?
\end{question}

The following question suggests a possible answer to \cref{QF_n}:

\begin{question}
Is the converse of \cref{FreeUHCEg} true, up to applying an automorphism of~$F_n$ to~$s$? What if we remove the assumption that the hamiltonian circle is unique?
\end{question}

We establish a positive answer to this question in two special cases, the second 
of which gives a complete solution to \cref{QF_n} when $n = 2$:

\begin{restatable}{thm}{forthmain}\label{F_2}
Let $s$ be an element of a symmetric, finite generating set~$S$ of~$F_n$ \textup(with $n \ge 2$\textup), and assume that either
    \begin{enumerate} 
    \renewcommand{\theenumi}{\alph{enumi}}
    \item \label{F_2-degree}
    $\#_a^\pm(s) = 2$ for all $a \in A$,
    or
    \item \label{F_2-n=2}
    $n = 2$.
    \end{enumerate}
The $2$-factor\/ $\cay(F_n; s^{\pm1})$ is a hamiltonian circle in $\cay(F_n; S)$ if and only if some automorphism of~$F_n$ carries $s$ to one of the two words listed in the statement of \cref{FreeUHCEg}.
\end{restatable}

It is well known that the cube of every connected, finite graph has a hamiltonian cycle. By using hamiltonian circles in the place of hamiltonian cycles, this was extended by A.\,Georgakopoulos \cite[Thm.~5]{MR2483226} to all locally finite connected graphs. This implies that if $S_0$ is any finite generating set of~$G$, then $\cay(G; S_0^{\pm1} \cup S_0^{\pm2}\cup S_0^{\pm3})$ always has a hamiltonian circle. By letting $S_0$ be the standard generating set, this implies there is a symmetric generating set $S$ of $F_n$, such that $\cay(F_n; S)$ has a hamiltonian circle, and $\#S = O(n^3)$. A.\,Georgakopoulos [personal communication] conjectured that $n^3$ can be improved to~$n$. \Cref{FreeUHCEg} establishes an explicit form of this conjecture:

\begin{cor} \label{SmallGenSetFn}
There is a symmetric generating set~$S$ of~$F_n$, such that $\#S = 2n+2$ and $\cay(F_n; S)$ has a hamiltonian circle.
\end{cor}

\begin{remark}
$2n + 2$ is best possible for $n \ge 2$: it is easy to see that no symmetric set of cardinality less than~$2n$ can generate~$F_n$, and it is well known that a symmetric generating set of cardinality exactly~$2n$ must consist of a free generating set and its inverses \cite{SmallGenSetOfFn}, in which case the corresponding Cayley graph is a tree and therefore obviously does not have a hamiltonian circle.
\end{remark}

It may be of interest to note that \cref{SmallGenSetFn} implies a (weaker) bound for groups that are virtually free:

\begin{restatable}{cor}{VirtFree} \label{VirtFree}
If $G$ is a group with a free subgroup $F_n$ of index~$m$, then $G$ has a symmetric generating set~$S$ of cardinality $2m + 4n + 2$, such that $\cay(G; S)$ has a hamiltonian circle. Furthermore, if  $F_n$ is normal, then $2m$ can be replaced with $2 + 2 \log_2 m$.
\end{restatable}

Here is an outline of the paper. 
\Cref{PrelimSect} collects some preliminaries on vertex-transitive graphs, Cayley graphs, hamiltonian circles, and other basic topics.
\Cref{finmanyends} proves the results that were stated above on vertex-transitive graphs with finitely many ends. (It also provides an explicit description of the Cayley graphs with finitely many ends that have a unique hamiltonian circle or a unique two-way-infinite hamiltonian path.)
\Cref{LeggeSect} gives a rigorous construction of Legge's examples~\pref{LeggeEg}.
\Cref{FreeSect} proves \cref{FreeUHCEg,VirtFree} (and some other results on free groups that were not stated above). 
Finally, \Cref{F2Sect} proves \cref{F_2}.

\wantpicturefalse

\section{Preliminaries} \label{PrelimSect}

\subsection{Cyclic groups and dihedral groups}

\begin{notation} \label{ZDNotation}
\leavevmode
	\begin{enumerate}
	\item \label{ZDNotation-n}
	For $n \in \Z^+$:
		\begin{enumerate}
		\item $\Z_n$ is the (additive) cyclic group of integers modulo~$n$,
		and
		\item \label{ZDNotation-n-D}
		$D_{2n} = \langle\, a, b \mid a^2 = b^2 = (ab)^n = 1 \,\rangle$ is the dihedral group of order~$2n$.
		\end{enumerate}
	\item \label{ZDNotation-Dinfty}
	$D_\infty = \langle\, a, b \mid a^2 = b^2 = 1 \,\rangle$ is the infinite dihedral group.
	\end{enumerate}
\end{notation}

\subsection{Vertex-transitive graphs and Cayley graphs}

\begin{definition}[{\cite[p.~33]{godsil2001algebraic}}] \label{VertTransDefn}
A graph~$X$ is \emph{vertex-transitive} if the automorphism group of~$X$ acts transitively on the vertex set of~$X$: for all $x,y \in V(X)$, there exists $\varphi \in \aut$, such that $\varphi(x) = y$.
\end{definition}

\begin{definition} \label{CayleyDefn}
(cf. \cite[p. 34]{godsil2001algebraic}). If $S$ is a symmetric subset of a group~$G$, then the corresponding \emph{Cayley
graph} $\cay(G;S)$ is the undirected graph whose vertices are the elements of $G$, and such that vertices
$x$ and $y$ are adjacent if and only if $x^{-1}y \in S$. (A set~$S$ is \emph{symmetric} if $S = S^{-1}$, where  $S^{-1} = \{\, s^{-1}\mid s\in S \,\}$.) All Cayley graphs are vertex-transitive, because the left translation $x \mapsto gx$ is an automorphism for every $g \in G$. 
\end{definition}

\begin{definition}[cf.\ {\cref{2wayHPDefn}}] 
The \emph{two-way-infinite path}~$P_\infty$ is the graph $\cay(\Z; \pm 1)$. If $X$ is a graph, then any subgraph of~$X$ that is isomorphic to~$P_\infty$ is called a \emph{two-way-infinite path in~$X$}.
\end{definition}

\begin{notation}[cf.~{\cite[\S2.1]{WitteGallian-survey}}]
 If we fix a starting vertex~$v$, then the sequence  $(s_1, s_2, \ldots , s_m)$ of elements of $S^{\pm1}$ represents the walk in $\cay(G;S)$
that visits the vertices
$$v, vs_1, vs_1s_2, \ldots , vs_1s_2 \ldots s_m.$$
Also, we use $s^{\ell}$ or $s^{-\ell}$
to represent the sequences $(s,\ldots,s)$ and $(s^{-1},\ldots,s^{-1})$ of
length $\ell$.
\end{notation}

\begin{definition}(cf \cite[p 19]{hammack2011handbook})\label{Sabidussi}
An action of a group~$G$ on a set~$V$ is \emph{sharply transitive} if it is both transitive and free. This means that, for all $v,w \in V$, there is a unique $g \in G$, such that $gv = w$.
\end{definition}

\begin{lem}[Sabidussi's Theorem {\cite[Theorem 1.2.20]{dobson2022symmetry}}]\label{sabidussi}
Let $X$ be a graph, and let $G$ be a group. 
There is a subset~$S$ of~$G$, such that $X$ is isomorphic to the Cayley graph\/ $\cay(G; S)$ if and only if\/ $\aut$ contains a subgroup that is isomorphic to~$G$ and acts sharply transitively on $V(X)$.
\end{lem}

The following notation is intended to make it easier for the reader to verify that $g$ is indeed adjacent to~$h$ in $\cay(G; S^{\pm1})$.

\begin{notation} \label{gtogs}
If $h = gs$ and $s \in S$, then we often use $g \stackrel{s}{\to} h$ or $h \stackrel{s}{\from} g$ to denote the edge $g \edge h$ in $\cay(G; S^{\pm1})$.
\end{notation}

\subsection{Edge cuts and hamiltonian circles}

\begin{definition}[cf.~{\cite[p.~1426]{MR2800969}}]
Let $X$ be a graph.
    \begin{enumerate}
    \item For $A \subseteq V(X)$, we let $\bdry_X(A)$ be the set of all edges of~$X$ that have one endpoint in~$A$ and the other endpoint in the complement of~$A$. 
    \item A set~$\edgecut$ of edges of~$X$ is an \emph{edge cut} if there is a nonempty, proper subset~$A$ of $V(X)$, such that $\edgecut = \bdry_X(A)$.
    \item An edge cut is \emph{finite} if it contains only finitely many edges.
    \end{enumerate}
\end{definition}

\begin{lem}[{\cite[Thm.~3]{kundgen2017cycles}}]\label{2factor}
Let $X$ be a locally finite graph. Then every hamiltonian circle~$C$ in~$X$ is a $2$-factor of~$X$, such that:
\begin{enumerate}
    \item \label{2factor-even}
    every finite edge cut contains a positive even number of edges of~$C$, and
    \item \label{2factor-2edges}
    for all distinct $e_1, e_2\in E(C)$, there is a finite edge cut $\edgecut$ of~$X$, such that $E(C) \cap \edgecut = \{e_1, e_2\}$.
\end{enumerate}
 Conversely, if a $2$-factor $C$ satisfies \pref{2factor-even} and~\pref{2factor-2edges}, then $C$ is a hamiltonian circle.
\end{lem}

The following observation is presumably known, but we provide a proof because we do not have a reference. Our argument derives it from \cref{2factor}, but it is also a consequence of the fact \cite[pp.~145--146]{BridsonHaefliger} that the set of ends of a group is independent of the choice of a finitely generated Cayley graph of the group.

 \begin{cor} \label{ChangeS}
Let $S$ be a finite, symmetric generating set of the group~$G$, and let $s \in S$, such that $\cay(G; s^{\pm1})$ is a hamiltonian circle in $\cay(G; S)$. If $\varphi$ is any group automorphism of~$G$, and $S_0$ is any finite, symmetric generating set of~$G$ that contains $\varphi(s)$, then $\cay \bigl( G; \varphi(s)^{\pm1} \bigr)$ is a hamiltonian circle in $\cay(G; S_0)$.
\end{cor}

\begin{proof}
The conditions~\pref{2factor-even} and~\pref{2factor-2edges} of \cref{2factor} are purely graph theoretic, so they are preserved by all graph isomorphisms. Therefore, $\cay \bigl( G; \varphi(s)^{\pm1} \bigr)$ is a hamiltonian circle in $\cay \bigl( G; \varphi(S) \bigr)$.

To complete the proof, we show that $\cay \bigl(G; \varphi(s)^{\pm1})$ remains a hamiltonian circle when $\varphi(S)$ is replaced with any other finite, symmetric generating set $S_0$ that contains~$\varphi(s)$. For convenience, let $C = \cay \bigl(G; \varphi(s)^{\pm1})$, $X = \cay \bigl( G; \varphi(S) \bigr)$, and $X_0 = \cay(G; S_0)$.
We wish to show that conditions~\pref{2factor-even} and~\pref{2factor-2edges} of \cref{2factor} are satisfied with $X_0$ in the place of~$X$.

For completeness, we provide a proof of the well known fact that, for all $A \subseteq G$, the edge cut $\delta_X(A)$ is finite if and only if $\delta_{X_0}(A)$ is finite. Since $X$ and~$X_0$ are Cayley graphs of the same group, they are quasi-isometric \cite[Eg.~8.17(3), p.~139]{BridsonHaefliger}, so there is some~$k$, such that if $e \in \bdry_{X_0}(A)$, then the two endpoints of~$e$ are joined by a path in~$X$ whose length is less than~$k$. This implies that both endpoints of~$e$ are within distance~$k$ of some edge in $\delta_X(A)$. Since $X$ is locally finite, we conclude that if $\delta_X(A)$ is finite, then $\delta_{X_0}(A)$ is also finite. The converse is similar.

We now use this fact to establish~\pref{2factor-even} and~\pref{2factor-2edges}.

\pref{2factor-even} Suppose $\edgecut_0 = \bdry_{X_0}(A)$ is a finite edge cut of $\cay(G; S_0)$. Then $\bdry_{X}(A)$ is a finite edge cut of~$X$, so \fullcref{2factor}{even} tells us that $\bdry_C(A)$ is a positive, even number. Therefore $\edgecut_0$ contains a positive, even number of edges of~$C$.

\pref{2factor-2edges} Let $e_1$ and $e_2$ be two distinct edges of~$C$. By \fullcref{2factor}{2edges}, there is a finite edge cut $\edgecut = \bdry_X(A)$ of~$X$, such that $\bdry_C(A) = \{e_1, e_2\}$. Then $\edgecut_0 = \bdry_{X_0}(A)$ is a finite edge cut of~$X_0$, and we have $E(C) \cap \edgecut_0 = \bdry_C(A) = \{e_1, e_2\}$.
\end{proof}

\begin{lem} \label{UniqueInCayley}
If $C$ is a unique hamiltonian circle in\/ $\cay(G; S)$, then either:
\noprelistbreak
	\begin{enumerate}
	\item \label{UniqueInCayley-s}
 there is an element~$s$ of~$S$, such that $C = \cay(G: s^{\pm 1})$, 
	or
	\item \label{UniqueInCayley-ab}
 there are elements $a$ and~$b$ of order~$2$ in~$S$, such that $C = \cay(G; a,b)$.
	\end{enumerate}
If $G$ is finite, then it must be the case that $\langle s \rangle = G$, or $\langle a,b \rangle = G$, respectively.
\end{lem}

\begin{proof}
The uniqueness of~$C$ implies that it is invariant under all automorphisms of $\cay(G; S)$. In particular, it must be invariant under the left-translation by each element of~$G$. From this, we see that if $a$ and $b$ are the two neighbours of~$1$ in~$C$, then, for every $v \in G$, the vertices $va$ and $vb$ are the neighbours of~$v$ in~$C$; therefore $C = \cay(G; S)$, where $S = \{a,b\}$. Since $S$ must be symmetric, this implies that either $b = a^{-1}$, so \pref{UniqueInCayley-s}~is satisfied with $s = a$, or $|a| = |b| = 2$, so \pref{UniqueInCayley-ab}~is satisfied.
\end{proof}

\section{Graphs with finitely many ends}\label{finmanyends}

This \lcnamecref{finmanyends} proves the results that were stated in the Introduction and relate to vertex-transitive graphs with finitely many ends that have a unique hamiltonian circle:
\begin{itemize}
    \item \cref{FiniteCase}: the only finite ones are cycles;
    \item \cref{2wayUHP}: $P_\infty$ is the only vertex-transitive graph with a unique two-way-infinite hamiltonian path;
    and
    \item \cref{2endedUHC}: the only $2$-ended ones are the square of~$P_\infty$ and the two-way-infinite ladder.
\end{itemize}
(For the reader's convenience, the complete statement of each of these results is reproduced in this section.)
Also, although these results were not mentioned in the Introduction, we give an explicit description of each group~$G$ and symmetric generating set~$S$, such that the Cayley graph $\cay(G;S)$ has:
    \begin{itemize}
        \item $0$ ends and a unique hamiltonian circle (\cref{FiniteCayley});
        or
        \item a unique two-way-infinite hamiltonian path (\cref{CayUnique2way});
        or
        \item $2$ ends and a unique hamiltonian circle (\cref{2endedCayUHC}).
    \end{itemize}


\FiniteCase*

\begin{proof}
Let $X$ be a finite vertex-transitive graph that has a unique hamiltonian cycle~$C$, and let $n = |V(X)|$. 

We claim that $X$ is (isomorphic to) a Cayley graph of either $\Z_n$ or $D_n$. First, note that the uniqueness implies that $C$ is invariant under every automorphism of~$X$, so $\aut \subseteq \autt(C) = D_{2n}$. 
Let $R$ be the group of rotations in~$D_{2n}$ (so $R \cong \Z_n$). If $R \subseteq \aut$, then $X$ is a Cayley graph of $R$, by Sabidussi's Theorem~\pref{Sabidussi}. On the other hand, if $R \not\subseteq \aut$, then, since $\aut$ is a transitive subgroup of~$D_{2n}$, it is easy to see that $\aut \cong D_n$, and that $\aut$ acts sharply transitively on $V(X)$. Therefore, $X$ is a Cayley graph of $D_n$, by Sabidussi's Theorem~\pref{Sabidussi}.
This completes the proof of the claim.

Thus, we have two cases to consider.

\refstepcounter{caseholder} 

\begin{case} \label{FiniteVTpf-Zn}
Assume $X$ is a Cayley graph of~$\Z_n$.
\end{case}
Write $X = \cay(\Z_n; S)$.
From \cref{UniqueInCayley}, we see that we may assume the hamiltonian cycle is $\cay(\Z_n; \pm 1)$, so $\{\pm1\} \subseteq S$. 
If $X$ is not a cycle, then $S$ must contain a third element~$s$. Assume, without loss of generality, that $2 \le s \le n - 2$. Then $(s, -1^{s-1}, s, 1^{n-s - 1})$,  is a (well known) hamiltonian cycle in~$X$. 
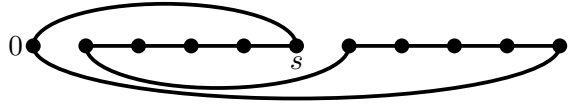
\[
\begin{tikzpicture}
        \begin{scope}[scale=0.7]
        \draw[line width=0.5mm] (0,0) arc(180:0:2.5cm and 0.8cm);
        \draw[line width=0.5mm] (5,0) -- (1,0);
        \draw[line width=0.5mm] (1,0) arc(180:360:2.5cm and 0.8cm);
         \draw[line width=0.5mm] (6,0) -- (10,0);
	    \draw[line width=0.5mm] (0,0) arc(180:360:5cm and 1cm);
        \foreach \i in {0,...,10}{
	           \draw (\i,0) node [circle,fill, inner sep=2pt] {};}
        \draw (0,0) node [left] {$0$};
        \draw (5,0) node [below] {$s$};

        \end{scope}
        \end{tikzpicture} 
\]  
This contradicts the uniqueness of~$C$.

\begin{case}
Assume $X$ is a Cayley graph of~$D_n$.
\end{case}
Write $X = \cay(D_n; S)$.
From \cref{UniqueInCayley}, we see that the hamiltonian cycle is $\cay(D_n; a,b)$, where $a$ and~$b$ are reflections that generate~$D_n$. (So $a,b \in S$.)
If $X$ is not a cycle, then $X$ must contain a third element~$c$. 

If $c$ is a rotation, then $c = (ab)^k$, for some~$k$ with $1 \le k < n/2$. It is not difficult to see that $X \cong \cay(\Z_n; \pm 1, \pm 2k)$, so \cref{FiniteVTpf-Zn} applies.

Therefore, we may assume that $c$ is a reflection. Then $\cay(D_n; a,b,c)$ is a cubic, spanning, hamiltonian subgraph of~$X$. It is well known that no finite cubic graph is uniquely hamiltonian (see \cref{smith}), so this spanning subgraph must have a second hamiltonian cycle~$C'$. Since $C'$ is also a hamiltonian cycle in~$X$, this contradicts the uniqueness of~$C$.
\end{proof}

\begin{cor} \label{FiniteCayley}
A finite Cayley graph
$\cay(G; S)$ has a unique hamiltonian cycle if and only if, up to a group isomorphism, it is either\/
        $\cay(\Z_n;  \pm 1)$ or\/ $\cay(D_{2n};a,b)$
        \textup(where $a$ and~$b$ are the generators in \fullcref{ZDNotation}{n-D}\textup).
\end{cor}



\wayUHP*
\begin{proof}
Let $X$ be a vertex-transitive graph that has a unique two-way-infinite hamiltonian path
	\[ P = \cdots \edge x_{-2} \edge x_{-1} \edge x_{0} \edge x_{1} \edge x_{2} \edge \cdots . \]
We may assume $X \neq P$, so there is an edge of~$X$ that is not in~$P$. By renumbering, we may assume the edge is of the form $(x_0, x_k)$, where $k > 1$.

We claim that, for all $i \in \Z$: 
	\[ \text{$(x_i, x_{i + k})$ is an edge of~$X$ if either $i$ is even or~$k$ is even.} \]
From the uniqueness of~$P$, we have $\aut \subseteq \autt P = D_\infty$. Since $\aut$ is transitive on $V(X)$, this implies the translation $x_n \mapsto x_{n + 2}$ is an automorphism of~$X$. By applying the $(i/2)$th power of this translation to the edge $(x_0, x_k)$, this establishes the claim in the case where $i$ is even.
So we may assume $k$ is even. Then $-k$ is even, so, by the case that has already been established, $(-k, 0)$ is an edge of~$X$; hence $x_0$ is adjacent to both $x_k$ and~$x_{-k}$. By vertex-transitivity, we conclude that $x_i$ is adjacent to both $x_{i+k}$ and~$x_{i-k}$. In particular, it is adjacent to $x_{i + k}$. This completes the proof of the claim.

The claim implies that the graph $X$ has the two-way-infinite hamiltonian path $(k, -1^{k-2}, k, -1)^\infty$. Here is an illustration of the case where $k = 5$:
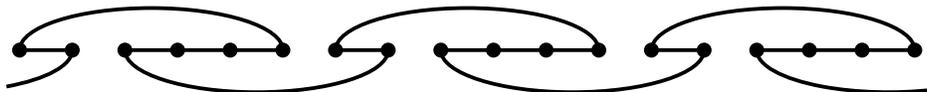
\[
\begin{tikzpicture}
        \begin{scope}[scale=0.7]
        \draw[line width=0.5mm] (0,0) arc(180:0:2.5cm and 0.8cm);
        \draw[line width=0.5mm] (6,0) arc(180:0:2.5cm and 0.8cm);
        \draw[line width=0.5mm] (12,0) arc(180:0:2.5cm and 0.8cm);
        \draw[line width=0.5mm] (2,0) arc(180:360:2.5cm and 0.8cm);
        \draw[line width=0.5mm] (8,0) arc(180:360:2.5cm and 0.8cm);
        \draw[line width=0.5mm] (14,0) arc(180:290:2.5cm and 0.8cm);
        \draw[line width=0.5mm] (1,0) arc(360:300:2.5cm and 0.8cm);
        \draw[ line width=0.5mm] (0,0)--(1,0);
        \draw[ line width=0.5mm] (2,0)--(5,0);
        \draw[ line width=0.5mm] (7,0)--(6,0);
        \draw[ line width=0.5mm] (8,0)--(11,0);
        \draw[ line width=0.5mm] (12,0)--(13,0);
        \draw[ line width=0.5mm] (14,0)--(17,0);

	\foreach \i in {0,...,17}{
	\draw (\i,0) node [circle,fill, inner sep=2pt] {};}
        \end{scope}
        \end{tikzpicture} 
\]  
This contradicts the uniqueness of~$P$.
\end{proof}

\begin{cor} \label{CayUnique2way}
A Cayley graph $\cay(G; S)$ has a unique two-way-infinite hamiltonian path if and only if \textup(up to an isomorphism of groups\textup) either
\noprelistbreak
	\begin{itemize}
	\item $G = \Z$ and $S = \{\pm1\}$, 
	or
	\item $G = D_{\infty}$ and $S = \{a, b\}$ \textup(where $a$ and~$b$ are the generators in \fullcref{ZDNotation}{Dinfty}\textup).
	\end{itemize}
\end{cor}

The following \lcnamecref{label2Ended} will be used in the proof of \cref{2endedUHC}.

\begin{lem} \label{label2Ended}
Let $X$ be a two-ended, vertex-transitive graph with a unique hamiltonian circle $C = P_1 \cup P_2$, where $P_1$ and~$P_2$ are two-way-infinite paths \textup(see \cref{hamin2ended}\,\textup).
Then there are 
	an automorphism~$g_2$  of~$X$ 
	and 
	a labelling of the vertices of~$X$,
 such that 
	\[ P_1 = \cdots \edge x_{-1} \edge x_{0} \edge x_{1} \edge \cdots,
	\qquad	
	P_2 = \cdots \edge y_{-1} \edge y_{0} \edge y_{1} \edge \cdots
	, \]
and
	\[ \text{$g_2(x_i) = x_{i+2}$ \quad and \quad $g_2(y_i) = y_{i+2}$ \quad for all $i \in \Z$} .
 \]
\end{lem}

\begin{proof}
By assumption, the Freudenthal compactification~$|X|$ has two points at infinity; let us call them $+\infty$ and~$-\infty$. By \cref{hamin2ended}, we may choose a labelling so that 
	\[ \lim_{n \to +\infty} x_n = \lim_{n \to +\infty} y_n = +\infty 
	\text{\quad and\quad}
	\lim_{n \to -\infty} x_n = \lim_{n \to -\infty} y_n = -\infty
	. \]

Since the setwise stabilizer of $V(P_1)$ acts on~$P_1$ via a transitive subgroup of $\autt(P_1) \cong D_\infty$, this stabilizer must contain an element~$g$ that acts on $V(P_1)$ via the translation by~$2$: i.e., $g(x_i) = x_{i+2}$ for all~$i$. 

 Then $\lim_{n \to +\infty} g^n x_0 = +\infty$. 
Since
	$d(g^n y_0, g^n x_0) = d(y_0, x_0) $
is constant (hence bounded, independent of~$n$), this implies $\lim_{n \to +\infty} g^n y_0 = +\infty$. 
So $g$ must also act on~$P_2$ by a nontrivial positive translation: $g(y_i) = y_{i + k}$, for some fixed $k > 0$.

To complete the proof, we will show that $k = 2$. Since $X$ is vertex-transitive (and $C$ is invariant under every automorphism), there is an automorphism~$\alpha$ of~$X$ that interchanges $P_1$ and~$P_2$. Conjugating $g$ by~$\alpha$ (and also passing to the inverse, if $\alpha$ interchanges $+\infty$ and~$-\infty$) yields $h \in \aut$, such that 
	\[ \text{$h(x_i) = x_{i + k}$ \quad and \quad $h(y_i) = y_{i + 2}$ \quad for all~$i$} . \]
Let $a = g^k h^{-2}$. Then $a$ fixes every vertex of~$P_1$, and acts as the translation by $k^2 - 4$ on~$P_2$. Since $X$ is connected, it must have some edge of the form $x_i \edge y_j$. Since $a$ is an automorphism that fixes $x_i$, then $x_i$ must be adjacent to every vertex in the $a$-orbit of~$y_j$, so the $a$-orbit of $y_j$ must be finite (since $X$ is locally finite). Therefore, $a$ cannot act on~$P_2$ by a nontrivial translation, so $k^2 - 4$ must be~$0$. Since $k > 0$, this implies $k = 2$, which completes the proof.
\end{proof}

\secondmain*

\begin{proof}
It is obvious that both of these graphs have $2$ ends, and (using \cref{hamin2ended}) it is not difficult to see that each of them has a unique hamiltonian circle. Therefore, we will focus on proving that there are no other such graphs.

Let $X$ be a two-ended, vertex-transitive graph with a unique hamiltonian circle~$C$, and write $C = P_1 \cup P_2$ as in \cref{hamin2ended}. Then \cref{label2Ended} provides an element $g_2$ of $\aut$, such that $g_2 (x_i) = x_{i+2}$ and $g_2 (y_i) = y_{i+2}$ for all~$i$.

Let $X_1$ be the subgraph of~$X$ induced by $V(P_1)$. This is vertex-transitive (since $X$ is vertex-transitive, and $C$ is invariant under every automorphism of~$X$). Hence, if $X_1 \neq P_1$, then we see from \cref{2wayUHP} that $X_1$ has a two-way-infinite hamiltonian path~$P_1'$ that is not equal to~$P_1$. Then $P_1' \cup P_2$ is a hamiltonian circle in~$X$, which contradicts the uniqueness of~$C$. Thus, we conclude that $P_1$ (and, similarly~$P_2$) is an induced subgraph of~$X$.
Therefore, if we let $r$ be the degree of~$X$, then every vertex has exactly $r - 2$ neighbours that are in the opposite component of~$C$.

We claim that, for every edge of the form $(x_i, y_j)$, there is an edge of the form $(x_{i+1}, y_k)$, such that $j - k$ is odd. By shifting the labellings of $P_1$ and~$P_2$, there is no harm in assuming $i = j= 0$. For convenience, let us say that a vertex of~$X$ is \emph{even} if it is of the form $x_{2m}$ or~$y_{2m}$, and that it is \emph{odd} otherwise. Now, if the desired edge does not exist, then every neighbour of~$x_1$ is even. Hence, if $y_{2m}$ is any even vertex in~$P_2$, then each of the $r - 2$ edges of the form $(x_1, y_k)$ can be translated by a power of~$g_2$ to an edge that is incident with~$y_{2m}$. So $y_{2m}$ has $r - 2$ neighbours that are odd vertices in~$P_1$. This accounts for all of the neighbours of~$y_{2m}$ that are in~$P_1$, so $y_{2m}$ is not adjacent to~$x_0$. Since $y_{2m}$ is an arbitrary even vertex of~$P_2$, we conclude that no neighbour of~$x_0$ is an even vertex. Then, since $y_i$ is a neighbour of~$x_0$, we see that $i$ is odd. Hence, we may choose $(x_1, y_k)$ to be any edge from~$x_1$ to a vertex in~$P_2$. This completes the proof of the claim.

We now claim that, more precisely, $(x_{i+1}, y_{j+1})$ is an edge. To see this, let $X'$ be the spanning subgraph of~$X$ that consists of $P_1 \cup P_2$, plus the orbits of the edges $(x_i, y_j)$ and $(x_{i+1}, y_k)$ under~$\langle g_2 \rangle$. Since $j - k$ is odd, every vertex of~$X'$ has degree~$3$. 
For simplicity, assume $i = j = 0$ (so $k$ is odd), by shifting the labellings of $P_1$ and~$P_2$. (So we wish to show that $k = 1$.)
We may also assume $k > 0$: if $k < 0$, then, after interchanging $P_1$ and~$P_2$, we have edges $(x_0, y_0)$ and $(x_k, y_1)$; since $k$ is odd, applying of power of~$g_2$ to the second edge yields the edge $(x_1, y_{1- k})$, and $1 - k > 0$. 
However, if $k > 1$, then it is not difficult to see that there is a second hamiltonian circle~$C'$ in~$X'$ (which contradicts the uniqueness of~$C$). 
One of the two-way-infinite paths in~$C'$ is 
	\[ \ldots, x_0, \ y_k, y_{k-1}, x_{k-1}, x_k, \ y_{2k}, y_{2k-1}, x_{2k-1}, x_{2k}, \ y_{3k}, \ldots , \]
and the other is
	\begin{align*}
	\ldots, \ &y_{k-2}, y_{k-3}, \ldots, y_1, x_1, x_2, \ldots, x_{k-2}, 
		\\ &y_{2k-2}, y_{2k-3}, \ldots, y_{k+1}, x_{k+1}, x_{k+2}, \ldots, x_{2k-2}, 
		\\ &y_{3k-2}, y_{3k-3}, \ldots, y_{2k+1}, x_{2k+1}, x_{2k+2}, \ldots, x_{3k-2}, 
		\ \ldots
	\end{align*}
Here is an illustration (for $k = 7$):
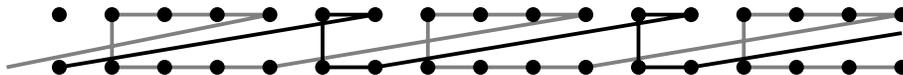
\[ \begin{tikzpicture}
    \begin{scope}[scale=0.7]

    \draw[line width=0.5mm,gray] (-8,0)--(-3,1)--(-6,1)--(-6,0)--(-3,0)--(3,1)--(0,1)--(0,0)--(3,0)--(9,1)--(6,1)--(6,0)--(9,0);

    \draw[line width=0.5mm] (-7,0)--(-1,1)--(-2,1)--(-2,0)--(-1,0)--(5,1)--(4,1)--(4,0)--(5,0)--(9,0.65);

    \foreach \i in {-7,...,9}{
    \draw (\i,1) node [circle,fill, inner sep=2pt] {};
    \draw (\i,0) node [circle,fill, inner sep=2pt] {};

    }
    \end{scope}
    \end{tikzpicture} \]
Hence, we must have $k = 1$, which completes the proof of the claim.

Let $g_1$ be the permutation of $V(X)$ that translates by~$1$ on both $P_1$ and~$P_2$: $g_1(x_n) = x_{n + 1}$ and $g_1(y_n) = y_{n+1}$. Then the last claim shows that the image of $(x_i, y_j)$ under~$g_1$ is an edge of~$X$. Since $(x_i, y_j)$ is an arbitrary edge joining $P_1$ and $P_2$, we conclude that $g_1$ is an automorphism of~$X$.

If $X$ is cubic, this implies that $X$ is the two-way-infinite ladder.
So we may assume there are edges from $x_0$ to $y_j$ and~$y_k$, where $j \neq k$. 

\refstepcounter{caseholder}

\begin{case}
Assume $j - k$ is even.
\end{case}
We may assume $0 = i = j < k$, by shifting the labellings of $P_1$ and~$P_2$, and perhaps interchanging $y_j$ and~$y_k$.
Then here is an illustration of a second hamiltonian circle, which contradicts the uniqueness of~$C$:
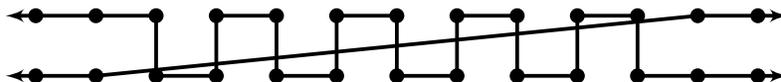
\[ \begin{tikzpicture}
    \begin{scope}[scale=0.8]

    \draw[<-,>=latex',line width=0.5mm] (-0.5,1)--(2,1);
    \draw[<-,>=latex',line width=0.5mm] (-0.5,0)--(1,0);
    \draw[->,>=latex',line width=0.5mm] (11,1)--(12.5,1);
    \draw[->,>=latex',line width=0.5mm] (10,0)--(12.5,0);
    \draw[,line width=0.5mm] (1,0)--(11,1);

    \foreach \i in {2,...,10}{
    \draw[,line width=0.5mm] (\i,1)--(\i,0);
    }
    \foreach \i in {2,4,6,8}{
    \draw[,line width=0.5mm] (\i,0)--(\i+1,0);
    }
    \foreach \i in {3,5,7,9}{
    \draw[,line width=0.5mm] (\i,1)--(\i+1,1);
    }
    \foreach \i in {0,...,12}{
    \draw (\i,1) node [circle,fill, inner sep=2pt] {};
    \draw (\i,0) node [circle,fill, inner sep=2pt] {};

    }
    \end{scope}
    \end{tikzpicture}\]
   
\begin{case}
Assume there is a third edge from $x_0$ to~$P_2$. 
\end{case}
Let us say there is an edge from $x_0$ to~$y_\ell$.
Then two of $j,k,\ell$ have the same parity, so the previous case applies. 

\begin{case}
Assume that $j - k$ is odd, and that the degree of~$X$ is~$4$. 
\end{case}
Let $\ell = j - k$ and assume, without loss of generality, that $\ell>0$. 

We claim that $X \cong \cay(\Z; \pm 2, \pm \ell)$. To establish this, assume, without loss of generality, that $k = 0$, and define $\varphi \colon V(X) \to \Z$ by
    \[ \text{$\varphi(x_m) = 2m$ 
    \quad and \quad
    $\varphi(y_n) = 2n - \ell$} . \]
The neighbours of~$x_m$ in~$X$ are $x_{m\pm1}$, $y_m$, and $y_{m + \ell}$; note that the image of these vertices under~$\varphi$ are $2m \pm 2$ and $2m \pm \ell$, which are precisely the neighbours of $\varphi(x_m) = 2m$ in the Cayley graph. Similarly, the neighbours of~$y_n$ in~$X$ are $y_{n\pm1}$, $x_n$, and $x_{n - \ell}$; note that the image of these vertices under~$\varphi$ are $2n - \ell \pm 2$, $2n$, and $2n - 2\ell$, which are precisely the neighbours of $\varphi(y_n) = 2n - \ell$ in the Cayley graph. So $\varphi$ is an isomorphism. This completes the proof of the claim.

From the claim, we see that if $\ell = 1$, then $X$ is the square of a two-way-infinite path.

Therefore, we may assume $\ell > 1$. (Since $\ell$ is odd, this means $\ell \ge 3$.) 
	\begin{itemize}
	\item The subgraph induced by each coset of~$\langle \ell \rangle$ is a two-way-infinite path,
	and
	\item for each $v \in \Z$, the path $v, v + 2, \ldots, v + 2(\ell - 1)$ contains exactly one element of each coset.
	\end{itemize}
Therefore, it is not difficult to see (and is well known) that $X$ contains a spanning subgraph that is isomorphic to the cartesian product $P_\ell \mathbin{\square} P_\infty$. It is not difficult to see that this subgraph has a second hamiltonian circle (which contradicts the uniqueness of~$C$). 
One of the two-way-infinite paths in the hamiltonian circle is~$(\ell)^\infty$, and the other is of the form $(2^{\ell-2}, \ell, -2^{\ell-2}, \ell)^\infty$.
\Cref{Cay(Z;23)} is an illustration of the case where $\ell=5$.
\end{proof}

\begin{figure}[ht]
		\centering
		\begin{tikzpicture}
	
	\tikzmath{\n1 = 5; 
		\n2 = \n1 - 2;
		\n3 = \n1 - 3;
		}

        \draw[line width=0.8mm,<->,>=latex'] (-3.5,1)--(4.5,1);

        \foreach \i in {-\n3,...,0}{
        		\draw[line width=0.8mm,<->,>=latex',] (-3.5,\i)--(4.5,\i);
        		\draw[line width=0.5mm,<->,>=latex',white] (-3.45,\i)--(4.45,\i);
		}

        \foreach \i in {-3,...,4}{
        \draw[line width=0.8mm] (\i,-\n1 + 2)--(\i,0);
        }
        
        \draw[line width=0.8mm,->,>=latex'] (4,-\n2)--(4.5,-\n2);
        \draw[line width=0.8mm,<-,>=latex'] (-3.5,-\n2)--(-3,-\n2);
        
       \foreach \i in {-3,...,4}{
        \draw[line width=0.8mm] (\i,0)--(\i,1);
        \draw[line width=0.5mm, white] (\i,0)--(\i,1);
        }     
       \foreach \i in {1,...,4}{
        \draw[line width=0.8mm] (2*\i - 5, 0)--(2*\i-4,0);
        \draw[line width=0.8mm] (2*\i - 5, -\n2)--(2*\i-4,-\n2);
        \draw[line width=0.5mm, white] (2*\i - 5, -\n2)--(2*\i-4,-\n2);
        }
       \foreach \i in {1,...,3}{
        \draw[line width=0.8mm] (2*\i - 4, -\n2)--(2*\i-3,-\n2);
        }
        
        \foreach \i in {-3,...,4}
        \foreach \j in {-\n2,...,1}{
        \draw (\i,\j) node [circle,fill, inner sep=2pt] {};
        }
 
    \end{tikzpicture}
    \caption{The black edges form a hamiltonian circle in the spanning subgraph $P_5 \mathbin{\square} P_\infty$ of $\cay(\Z; \pm 2, \pm 5)$.}
    \label{Cay(Z;23)}
    \end{figure}
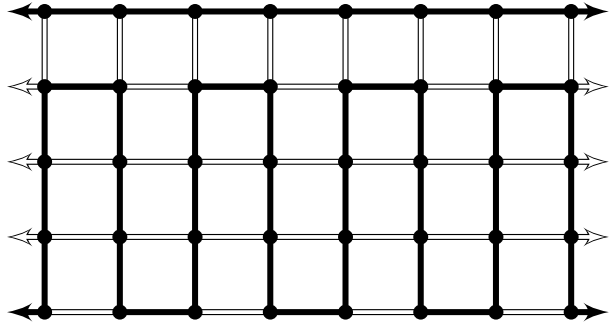

\begin{cor} \label{2endedCayUHC}
A Cayley graph $\cay(G;S)$ is two-ended and has a unique hamiltonian circle if and only if $G$ and~$S$ are in the following list \textup(up to a group isomorphism\textup{):}
	\begin{enumerate}
	\item square of the two-way-infinite path:
		\begin{enumerate}
		\item $G = \Z$ and $S = \{\pm 1, \pm 2\}$,
		\item $G = D_\infty$ and $S = \{a, b, ab, ba\} = \{a,b,(ab)^{\pm1}\}$, 
		\end{enumerate}
	\item two-way-infinite ladder:
		\begin{enumerate}
		\item $G = \Z \times \Z_2$ and $S = \{\pm (1,0), (0,1) \}$,
		\item $G = D_\infty$ and $S = \{a, ab, ba\} = \{a, (ab)^{\pm1}\}$,
		\item $G = D_\infty$ and $S = \{a, b, aba\}$,
		\item $G = D_\infty \times \Z_2$ and $S = \{(a,0), (b,0), (0,1)\}$, 
		\end{enumerate}
	\end{enumerate}
where $D_\infty = \langle \,a,b \mid a^2 = b^2 = 1 \,\rangle$, as in \fullcref{ZDNotation}{Dinfty}.
\end{cor}

\begin{proof}
By Sabidussi's Theorem~\pref{Sabidussi}, realizing a particular graph~$X$ as a Cayley graph amounts to finding a subgroup of $\aut$ that is sharply transitive on $V(X)$. More precisely, it follows from L.\,Babai's \cite[Lem.~3.1]{Babai1977} fundamental work on the {Cayley Isomorphism Problem} that determining the representations of~$X$ as a Cayley graph, up to group isomorphism, is equivalent to determining the sharply transitive subgroups of $\aut$, up to conjugacy in $\aut$. Therefore, in order to show that the list in the statement of the \lcnamecref{2endedCayUHC} is complete, it suffices to show that:
	\begin{enumerate}
	\item the automorphism group of the square of a two-way-infinite path has only two subgroups that are sharply transitive on the set of vertices,
	and
	\item the automorphism group of the two-way-infinite ladder has only four subgroups that are sharply transitive on the set of vertices.
	\end{enumerate}
(Let us point out that, although Babai's statement involves conjugacy classes of subgroups, rather than individual subgroups, conjugation is not an issue in the present situation. To see this, suppose that $H$ is a sharply transitive subgroup of $\aut$. Then, we have $\aut = HK$, where $K$ is the stabilizer of a vertex. Since the stabilizer of any vertex in $\aut$ has order $2$, this implies that the index of $H$ in $\aut$ is~$2$. Thus, every sharply transitive subgroup is normal, and therefore has no other conjugates.)

\refstepcounter{caseholder}

\begin{case}
Assume $X$ is the square of a two-way-infinite path $P_\infty = \cay(\Z; \pm 1)$.
\end{case}
Since the hamiltonian circle is unique, it is invariant under all automorphisms of~$X$. Therefore, $P_\infty$ is invariant under all of these automorphisms, since it is obtained from~$X$ by removing the edges of the hamiltonian circle. Hence, we have
	\[ \aut = \autt P_\infty = D_\infty = \langle\, a, b \mid a^2 = b^2 = 1 \,\rangle , \]
where $a$ and~$b$ act by reflections. We may assume that $a$ inverts the edge $0 \edge 1$ (and therefore does not fix any vertices), and that $b$ fixes the vertex~$0$.
Then the only sharply transitive subgroups of $\autt P_\infty$ are $\langle ab \rangle\cong \Z$ and $\langle a, bab \rangle\cong D_{\infty}$.
To see this, note that the abelianization of $G$ is $\mathbb Z_2\times \mathbb Z_2$ and so $G$ has three subgroups of index $2$.
However, the subgroup that contains $b$ has a nontrivial vertex stabilizer (since $b$ fixes a vertex), so it is not sharply transitive.

\begin{case}
Assume $X$ is the two-way-infinite ladder.
\end{case}
Let $V(X) = \{x_i\}_{i \in \Z} \cup \{y_i\}_{i \in \Z}$, where 
	\begin{itemize}
	\item the neighbours of $x_i$ are $x_{i + 1}$, $x_{i-1}$, and $y_i$,
	and
	\item the neighbours of $y_i$ are $y_{i + 1}$, $y_{i-1}$, and $x_i$.
	\end{itemize}
Then $\aut = D_\infty \times \Z_2$, where $D_\infty$ is the setwise stabilizer of~$\{x_i\}_{i \in \Z}$ and~$\{y_i\}_{i \in \Z}$, and $\Z_2$ interchanges these two sets. Assume that $a$ inverts the edge $x_0 \edge x_1$ (and therefore does not fix any vertices), and that $b$ fixes the vertex~$x_0$. Also, let $z$ be a generator of~$\Z_2$.

If $G$ is a sharply transitive subgroup of $\aut$, then it contains a subgroup of~$D_\infty$ that is sharply transitive on~$\{x_i\}_{i \in \Z}$. This subgroup must be either $\langle ab \rangle$ or $\langle a, bab \rangle$.
	\begin{itemize}
	\item If $G$ contains $\langle ab \rangle$, then $G$ is either $\langle ab, z \rangle$ or $\langle ab, az \rangle$. 
	\item If $G$ contains $\langle a, bab \rangle$, then $G$ is either $\langle a, bab, z \rangle$ or $\langle a, bab, bz \rangle$.
	\qedhere \end{itemize}
\end{proof}

\section{Legge's examples with infinitely many ends}\label{LeggeSect}

In this \lcnamecref{LeggeSect}, we present a rigorous construction of Legge's examples~\pref{LeggeEg} of uniquely hamiltonian vertex-transitive graphs with infinitely many ends.

\begin{observation}
Assume the notation of \cref{LeggeEg}, and let $v$ and~$w$ be two adjacent vertices of~$T_m$. It is clear that $T_m$ has:
    \begin{itemize}
    \item an automorphism $\alpha$ of order~$m$ that cyclically permutes the edges that are incident with~$v$,
    and
    \item an automorphism $\beta$ of order~$2$ that interchanges~$u$ and~$v$. 
    \end{itemize}
These naturally induce automorphisms $\widetilde\alpha$ and~$\widetilde\beta$ of~$X_m$.

The subgroup generated by $\widetilde\alpha$ and~$\widetilde\beta$ is the free product of $\langle\widetilde\alpha\rangle$ and $\langle\widetilde\beta\rangle$, and acts sharply transitively on~$V(X_m)$, so we see from Sabidussi's Theorem~\pref{Sabidussi} that $X_m$ is isomorphic to a Cayley graph of the free product $Z_m * Z_2$ of a cyclic group of order~$m$ and a cyclic group of order~$2$. 
\end{observation}

From this point of view, it is easy to generalize \cref{LeggeEg}, by replacing $2$ with any~$n$:

\begin{thm}[Legge, cf.\ \cref{LeggeEg}]\label{LeggeGeneral}
Fix $m \ge 3$ and $n \ge 2$, and let 
	\[ Z_m * Z_n = \langle\, a, b \mid a^m = b^n = 1 \,\rangle . \]
be the free product of a cyclic group of order~$m$ and a cyclic group of order~$n$.
Then the graph
    \[ X_{m,n} \coloneqq \cay \bigl( Z_m * Z_n; a^{\pm1}, (ab)^{\pm1} \bigr) \]
has a unique hamiltonian circle, namely, $\cay \bigl( Z_m * Z_n; (ab)^{\pm1} \bigr)$.
\end{thm}

Before proving this \lcnamecref{LeggeGeneral}, let us recall that a standard technique \cite[p.~1436]{MR2800969} for studying hamiltonian circles in an infinite graph~$X$ is to make a list $x_0, x_1, \ldots$ of the vertices of~$X$, and use this to construct (for each~$n$) a minor of~$X$, by contracting each component of $X \smallsetminus \{x_0, \ldots, x_n\}$ to a vertex. (Any loops that arise in the contraction are deleted, but all multiple edges are retained.) For the study of Cayley graphs of $Z_m * Z_m$, it is more convenient to slightly modify this construction:

\begin{definition} \label{requivDef}
For each $r \in \N^+$:
\begin{enumerate}
    \item We define an equivalence relation $\requiv{r}$ on $Z_m * Z_n$ by declaring that two reduced words are equivalent if they agree up to (and including) the $r$th $b$-string (or they are equal, and have less than~$r$ $b$-strings). 
    The equivalence class of a word $w \in Z_m * Z_n$ is denoted by $[w]_r$.
    \item For a symmetric subset~$S$ of~$Z_m * Z_n$, the quotient graph $\cay ( Z_m * Z_n; S)/{\requiv{r}}$ is obtained from $\cay ( Z_m * Z_n; S)$ by identifying all vertices that are in the same equivalence class. The quotient graph is allowed to have multiple edges, but we remove all loops.
\end{enumerate}
\end{definition}

The following crucial observation is elementary:

\begin{lem}[cf.~{\cite[p.~1436]{MR2800969}}]\label{compactness}
The edges in any finite edge cut of\/ $\cay(Z_m * Z_n;S)$ form an edge cut of\/ $\cay(Z_m * Z_n;S)/{\requiv{r}}$, for all sufficiently large~$r$. Conversely, the edges in any edge cut of\/ $\cay(Z_m * Z_n;S)/{\requiv{r}}$ form a finite edge cut of $\cay(Z_m * Z_n; S)$.
\end{lem}

\begin{proof}[{\bf Proof of \cref{LeggeGeneral}}]
For convenience, let $C = \cay \bigl( Z_m * Z_n; (ab)^{\pm1} \bigr)$.

\medbreak

(uniqueness)
Let $B$ be the set of reduced words that start with a single~$b$ (so either the word is~$b$, or it starts with $ba^k$ for some $k \neq 0$), and let $\overline{B}$ be the complement. Then the only edges of~$X_{m,n}$ that join a vertex of~$B$ to a vertex of~$\overline{B}$ are $b \stackrel{ab}{\from} a^{-1}$ and $ba^{-1} \stackrel{ab} \to b^2$. So these two edges form a finite edge cut. Therefore, every hamiltonian circle must contain both of these edges (see \fullcref{2factor}{even}).

Since translation on the left is an automorphism of the Cayley graph, this implies that every hamiltonian circle contains every left-translate of the edge $ba^{-1} \stackrel{ab} \to b^2$, which means that every hamiltonian circle contains $g \stackrel{ab} \to g \cdot a b$, for every~$g$. Thus, every hamiltonian circle contains the $2$-factor~$C$, and must therefore be equal to this $2$-factor (see \cref{2factor}).
.
\medbreak

(existence)
We prove by induction that the quotient graph $C / {\requiv{r}}$ is a cycle for all~$r$; then \cref{2factor} finishes the proof. For convenience, we will use $\to$ to denote $\stackrel{ab}{\to}$.

\emph{Base case.} For $r = 1$, we have
\begin{align*}
&[1]_1 
		\to [ab]_1=[aba^{-1}]_1 
		\to [ab^2]_1 =[ab^2a^{-1}]_1
		\to \cdots  
        \to [ab^{n-1}]_1 =[ab^{n-1}a^{-1}]_1
        \\
    \to ~&[a]_1\to [a^2 b]_1=[a^2ba^{-1}]_1 \to [a^2b^2]_1
        \to \cdots 
    \to [a^2b^{n-1}]_1 = [a^2b^{n-1}a^{-1}]_1 
    \\
    & \vdots \\
     \to ~&[a^{m-1}]_1 \to [b]_1= [ba^{-1}]_1 \to [b^2]_1= [b^2a^{-1}]_1 \to \cdots  
		\to [b^{n-1}]_1= [b^{n-1}a^{-1}]_1  
    \\
    \to ~& [1]_1
    . \end{align*}
So $C / {\requiv{1}}$ is a cycle.

\emph{Induction step.} Since no vertex of $X / {\requiv{r+1}}$ can have degree~$> 2$, it suffices to show that each vertex in $X / {\requiv{r}}$ expands to a path in $X / {\requiv{r+1}}$.

Assume $v$ is a reduced word that has exactly~$r$ $b$-strings (and ends with~$b$, so $v$ is the shortest word in its $r$-equivalence class). 
Then, for all $i$, we have
    \[ [va^i]_{r+1} \to [va^{i+1} b]_{r+1} \]
and, for all nonzero $i$ and~$j$, we have
    \[ [va^i b^j]_{r+1} = [va^i b^ja^{-1}]_{r+1} \to [va^i b^{j+1}]_{r+1} . \]
Therefore, by the same calculation as above, but starting at~$v$, instead of at~$1$, and stopping at $v a^{m-1}$, because the next term $vb$ has a different $r$th $b$-string, we have:
	\begin{align*}
		 [v]_{r+1} 
		\to ~&[vab]_{r+1} 
		\to [vab^2]_{r+1} 
		\to \cdots  
		\to [vab^{n-1}]_{r+1} 
        \to [va]_{r+1}
    \\ \to ~&[va^2 b ]_{r+1}
        \to [va^2 b^2 ]_{r+1}
		\to \cdots  
		\to [va^2b^{n-1}]_{r+1}
        \to [va^2]_{r+1}
        \\ & \vdots
        \\ \to ~& [va^{m-1} b ]_{r+1}
            \to [va^{m-1} b^2 ]_{r+1}
            \cdots
        \to [va^{m-1} b^{n-1}]_{r+1}
        \to [va^{m-1}]_{r+1}
		. \qedhere \end{align*}
\end{proof}

\section{Some hamiltonian circles in Cayley graphs of free groups} \label{FreeSect}

This \lcnamecref{FreeSect} investigates hamiltonicity and unique hamiltonicity of Cayley graphs of nonabelian free groups. 
(The main objective is to prove \cref{FreeUHCEg}, but we also establish \cref{VirtFree} and some results that were not stated in the Introduction.)
It will be helpful to fix some notation.

\begin{notation} \label{FnNotation}
In this \lcnamecref{FreeSect}:
    \begin{enumerate}
    \item $F_n$ is the free group of rank~$n$.
    \item We assume $F_n$ is nonabelian (i.e., $n \ge 2$).
    \item $A = \{a_1, \ldots, a_n\}$ is the standard generating set of~$F_n$. Any element of $A^{\pm1}$ is called a \emph{letter}. Thus, if $s \in F_n$ and $s = s_1 \cdots s_\ell$ is the representation of~$s$ as a reduced word, then $s_1$ is the \emph{first letter} of~$s$, and $s_\ell$ is the \emph{last letter} of~$s$.
    \item $\ell(g)$ is the \emph{word length} of an element~$g$ of~$F_n$. Thus, if $g = s_1 \cdots s_r$ is the representation of~$g$ as a reduced word, then $\ell(g) = r$.
    \item For $a \in A^{\pm1}$ and $g \in F_n$, we let $\#^{\pm}_a(g)$ be the number of occurrences of~$a$, plus the number of occurrences of~$a^{-1}$, in the representation of~$g$ as a reduced word.
    \item The symbol~$S$ always represents a symmetric, finite generating set of~$F_n$.
    \end{enumerate}
\end{notation}

\begin{lem} \label{MustHaveAll}
If $s \in F_{n-1}$, then there does not exist a symmetric, finite generating set~$S$ of~$F_n$, such that $\cay(F_n; s^{\pm1})$ is a hamiltonian circle in $\cay(F_n; S)$.
\end{lem}

\begin{proof}
By \cref{ChangeS}, we may assume that $S = A^{\pm1} \cup \{s^{\pm1}\}$. Then, since $s \in F_{n-1}$, the edge $1 \edge a_n$ is an edge cut of $\cay(F_n; S)$. This edge cut has only one edge, so it is impossible for any subgraph to meet it in a positive, even number of edges. Therefore, we see from \fullcref{2factor}{even} that the Cayley graph does not have a hamiltonian circle.
\end{proof}

\begin{question}
The proof of \cref{MustHaveAll} shows that if some automorphism of~$F_n$ takes $s$ to an element of~$F_{n-1}$, then the closure of $\cay(F_n; s^{\pm1})$ in $|\cay(F_n; S)|$ is not connected. Is the converse true? (By an argument similar to the proof of \fullcref{F_2}{n=2}, it can be shown that the converse is indeed true when $n = 2$.)
\end{question}

It is also important to know a few automorphisms of~$F_n$:

\begin{lem}[cf.\ \cite{AutFn}] \label{AutFn}
\leavevmode
\begin{enumerate}
    \item \label{AutFn-perm}
    If $\sigma$ is any permutation of $\{1,\ldots, n\}$, then there is an automorphism~$\varphi$ of~$F_n$, such that $\varphi(a_i) = a_{\sigma(i)}$ for all~$i$.
    \item \label{AutFn-sign}
    if $\epsilon_1, \ldots, \epsilon_n \in \{\pm1\}$, then there is an automorphism~$\varphi$ of~$F_n$, such that $\varphi(a_i) = a_i^{\epsilon_i}$ for all~$i$.
    \item \label{AutFn-mult}
    For all $g,h \in F_{n-1}$, there is an automorphism $\varphi$ of~$F_n$, such that $\varphi(a_n) = g a_n h$, and $\varphi(a_i) = a_i$ for $1 \le i \le n-1$.
\end{enumerate}
\end{lem}

\begin{cor}
If $s \in S$, such that\/ $\cay(F_n; s^{\pm1})$ is a hamiltonian circle in the Cayley graph\/ $\cay(F_n; S)$, then $\#^{\pm}_{a_i}(s) \ge 2$ for $1 \le i \le n$.
\end{cor}

\begin{proof}
Suppose $\#^{\pm}_{a_i}(s) \le 1$. By \fullcref{AutFn}{perm} (and \cref{ChangeS}), we may assume $i = n$. If $\#^{\pm}_{a_n}(s) = 0$, then $s \in F_{n-1}$, which contradicts \cref{MustHaveAll}. Therefore, we must have $\#^{\pm}_{a_n}(s) = 1$, so we may write $s = g a_n^\epsilon h$, where $g,h \in F_{n-1}$ and $\epsilon = \pm 1$. Then it is not difficult to see from \cref{AutFn} that there is an automorphism~$\varphi$ of~$F_n$, such that $\varphi(s) = a_1 \in F_{n-1}$. This is a contradiction (by \cref{MustHaveAll,ChangeS}).
\end{proof}

\begin{thm}[Georgakopoulos {[personal communication]}]\label{uniqueFn}
If $s \in F_n$, such that $\#^{\pm}_{a_i}(s) \le 2$ for $1 \le i \le n$, and\/ $\cay \bigl( F_n; A^{\pm 1} \cup \{s^{\pm 1}\} \bigr)$ has a hamiltonian circle, then this Cayley graph is uniquely hamiltonian. 
More precisely, $\cay( F_n; s^{\pm 1})$ is the only hamiltonian circle in this Cayley graph.
\end{thm}

\begin{proof}
Let $C$ be a hamiltonian circle, and suppose there is a vertex~$v$, such that the edge $v \edge vs$ is not in~$C$. (This will lead to a contradiction.) Since $\cay( F_n; A^{\pm 1} )$ is a tree, it has a unique path from~$v$ to~$vs$; let
	\[ v = v_0 \edge v_1 \edge \cdots \edge v_k = vs \]
be this path. So $v_i = v_{i-1} s_i$ for $1 \le i \le k$, where each $s_i$ is in~$A^{\pm 1}$, and $s = s_1 s_2 \cdots s_k$ (as a reduced word). 

We claim that $C$ contains the path from~$v$ to~$vs$ in $\cay( F_n; A^{\pm1})$. In other words, we claim that $C$ contains the edge $v_{i-1} \edge v_i$, for $1 \le i \le k$. Note that removing this edge disconnects $\cay( F_n; A^{\pm1} )$ (since $\cay( F_n; A^{\pm1})$ is a tree). 
Let $X_1$ and~$X_2$ be the two connected components of this edge-deleted graph. 
Since $\#_{s_i}^{\pm1}(s) \le 2$, the graph $\cay( F_n; s^{\pm1})$ has at most 2 edges joining $X_1$ and~$X_2$.  So $\cay \bigl( F_n; A^{\pm1} \cup \{s^{\pm1}\} \bigr)$ has at most 3 edges joining these two sets. However, one of these edges is $v \edge vs$, which is not in~$C$. Since $C$ must intersect each finite edge cut in a positive even number of edges (see \fullcref{2factor}{even}), this implies that $C$ contains both of the other edges, one of which is $v_{i-1} \edge v_i$. This completes the proof of the claim.

We know from \cref{MustHaveAll} that $\cay \bigl( F_n; A^{\pm1} \cup \{s_1^{\pm k}\} \bigr)$ does not have a hamiltonian circle, so $s \neq s_1^k$, so there is some~$i$ (with $1 \le i < k$), such that $s_i = s_1$ and $s_{i+1} \neq s_1$. By the claim that was proved in the preceding paragraph, we also know that $v_i$ is adjacent in~$C$ to the two vertices
	\[ \text{$v_{i-1} = v_i s_i^{-1} = v_i s_1^{-1} \notin \{v_i s, v_i s_1\}$ 
	\quad and \quad
	$v_{i+1} = v_i s_{i+1}\notin \{v_i s, v_i s_1\}$} . \]
Since $C$ is a 2-factor, these are the only two neighbours of~$v_i$ in~$C$, so $v_i$ is not adjacent to $v_i s$ in~$C$. It is also not adjacent to $v_i s_1$, so $C$ does not contain the path from~$v_i$ to~$v_i s$ in $\cay( F_n; A^{\pm1})$. This contradicts the claim that was proved in the preceding paragraph (with~$v_i$ in place of~$v$).
\end{proof}

The following \lcnamecref{lsimDef} and \lcnamecref{FnCut} are straightforward analogues of \cref{requivDef,compactness}.

\begin{definition} \label{lsimDef}
For each $\ell \in \N^+$:
\begin{enumerate}
    \item  We define an equivalence relation $\lsim{\ell}$ on $F_n$ by declaring that two reduced words are equivalent if they have the same initial segment of length~$\ell$. (A word of length less than $\ell$ is not equivalent to any other elements of the group.) The equivalence class of an element $w \in F_n$ is $[w]_{\ell}$.
    \item $\cay(F_n; S)/{\lsim{\ell}}$ is the quotient graph that is obtained from $\cay(F_n; S)$ by identifying all vertices that are in the same equivalence class. The quotient graph is allowed to have multiple edges, but we remove all loops.
\end{enumerate}
\end{definition}

\begin{lem}[cf.~{\cite[p.~1436]{MR2800969}}]\label{FnCut}
The edges in any finite edge cut of $\cay(F_n;S)$ form an edge cut of $\cay(F_n;S)/{\lsim{\ell}}$, for all sufficiently large~$\ell$. Conversely, the edges in any edge cut of $\cay(F_n;S)/{\lsim{\ell}}$ form a finite edge cut of $\cay(F_n;S)$.
\end{lem}

The next two \lcnamecref{DegInGr}s 
provide important information about these quotient graphs.

\begin{lem} \label{DegInGr}
Let $v,s \in F_n$ \textup(with $s \neq 1$\textup) and $r\in\N^+$.
\noprelistbreak
    \begin{enumerate}
        \item \label{DegInGr-less}
        If $\ell(v) < r$, then $[v]_r$ has degree~$2$ in $\cay(F_n; s^{\pm1})/{\lsim{\ell}}$.
        \item \label{DegInGr-r}
        If $\ell(v) = r$, and the last letter of~$v$ is~$a$, then $\#_a^{\pm}(s)$ is equal to the degree of\/ $[v]_r$ in $\cay(F_n; s^{\pm1})/{\lsim{\ell}}$.
    \end{enumerate}
\end{lem}

\begin{proof}
\pref{DegInGr-less} We have $[vs^{-1}]\stackrel{s}{\to} [v]_r \stackrel{s}{\to} [vs]_r$, and these are the only two edges that are incident with~$v$.

\pref{DegInGr-r} Each edge incident with $[v]_r$ is of the form $[w]_r \edge [wt]_r$, with $t \in \{s, s^{-1}\}$ and $w \lsim{r} v$, but $wt \not\lsim{r} v$ (since loops are not allowed). Since $w \lsim{r} v$ (and the last letter of~$v$ is~$a$), we have $w = v x$, for some $x \in F_n$, such that the first letter of~$x$ is not~$a^{-1}$. Write $x = x_1 x_2 \cdots x_k$ as a reduced word. Since $wt \not\lsim{r} v$, the letters $a, x_1, x_2, \ldots, x_k$ at the end of~$w$ must all cancel in the product $wt$, so if we write $t = t_1 t_2 \cdots t_r$ as a reduced word, then $t_1 t_2 \cdots t_{k+1} = (ax)^{-1}$. In particular, we must have $t_{k+1} = a^{-1}$. (Conversely, if $t \in \{s, s^{-1}\}$ and $t_{k+1} = a^{-1}$, then $t_k^{-1} \neq a^{-1}$, so
    \[ [v]_r \stackrel{t}{\to} [v t_{k}^{-1} t_{k-1}^{-1} \cdots t_1^{-1} \cdot t]_r \]
is an edge of $\cay(F_n; s^{\pm1})/{\lsim{\ell}}$.) Thus, the edges incident with $[v]_r$ are in one-to-one correspondence with the occurrences of $a^{-1}$ in either $s$ or~$s^{-1}$. So the degree of $[v]_r$ is $\#_{a^{-1}}^{\pm}(s) = \#_{a}^{\pm}(s)$, as claimed.
\end{proof}

\begin{lem} \label{X1s}
Let $s = s_1 s_2 \cdots s_r$ be a nontrivial reduced word in~$F_n$. Then $\cay(F_n; s^{\pm1})/{\lsim{1}}$ is isomorphic to the graph $X^s_1$ that is defined as follows:
    \[ V(X^s_1) = A^{\pm1} \cup \{1\} \]
and 
    \[ E(X^s_1)
    = \{1 \edge s_1\}
    \cup 
    \{ s_i^{-1} \edge s_{i+1} \mid 1 \le i < r \,\} 
    \cup \{ s_r^{-1} \edge 1 \}
    \quad \text{\textup(a multiset\textup)}
    . \]
More precisely, an isomorphism $\varphi \colon X^s_1 \to \cay(F_n; s^{\pm1})/{\lsim{1}}$ is defined by $\varphi(x) = [x]_1$ for $x \in V(X^s_1)$.
\end{lem}

\begin{proof}
First, note that the neighbours of~$1$ in $\cay(F_n; s^{\pm1})$ are $s$ and~$s^{-1}$, so the neighbours of~$[1]_1$ in $\cay(F_n; s^{\pm1})/{\lsim{1}}$ are $[s]_1 = [s_1]_1$ and $s^{-1} = [s_r^{-1}]_1$, so $X^s_1$ has the correct neighbours of~$1$.

For $1 \le i < r$, there is an edge in $\cay(F_n; s^{\pm1})/{\lsim{1}}$ 
    \[ \text{from 
    $[s_i^{-1}]_1
    = [s_i^{-1} s_{i-1}^{-1} \cdots s_1^{-1}]_1$
    to
    $[s_i^{-1} s_{i-1}^{-1} \cdots s_1^{-1}\cdot s]_1
    = [s_{i+1} s_{i+2} \cdots s_r]_1
    = [s_{i+1}]_1$}
    . \]
These are precisely the other edges of~$X^s_1$.
\end{proof}

\begin{notation}[cf.\ \cref{gtogs}]
For a nontrivial reduced word $s = s_1 s_2 \cdots r$, we often use $s_i^{-1} \stackrel{s_i s_{i+1}}{\to} s_{i+1}$ or $s_{i+1} \stackrel{s_i s_{i+1}}{\from} s_i^{-1}$ to denote the edge $s_i^{-1} \edge s_{i+1}$ of~$X_1^s$. Also, we use $1 \stackrel{s_1}{\to} s_1$ and $s_n^{-1} \stackrel{s_n}{\to} 1$ to denote the edges $1 \edge s_1$ and $s_n^{-1} \stackrel{s_n}{\to} 1$, respectively.
\end{notation}

Now, we describe a simple way to construct hamiltonian circles:

\begin{prop} \label{CycleToHamCircle}
    Let $s \in S$. If\/ $\cay(F_n; s^{\pm1})/{\lsim{1}}$ is a cycle, then $\cay(F_n; s^{\pm1})$ is a hamiltonian circle in $\cay(F_n; S)$.
\end{prop}

\begin{proof}
First, note that, for all~$i$:
    \begin{itemize}
    \item the degree of $[a_i]_1$ in $\cay(F_n; s^{\pm1})/{\lsim{1}}$ is~$2$ (because this graph is assumed to be a cycle),
    and
    \item the degree of $[a_i]_1$ in $\cay(F_n; s^{\pm1})/{\lsim{1}}$ is $\#_{a_i}^{\pm}(s)$ (by \fullcref{DegInGr}{r}).
    \end{itemize}
So we have $\#_{a_i}^{\pm}(s) = 2$ for all~$i$. Therefore, we see from \cref{DegInGr} that the graph $\cay(F_n; s^{\pm1})/{\lsim{\ell}}$ is regular of degree~$2$ for all $\ell \ge 1$.

We will show (by induction) that $\cay(F_n; s^{\pm1})/{\lsim{\ell}}$ is connected for all~$\ell$. Then the conclusion of the preceding paragraph implies that this quotient graph is a cycle for all $\ell \ge 1$. With this, the desired conclusion follows easily from \cref{2factor} (and \cref{FnCut}).

The base case of the induction is true by assumption (since cycles are connected). For the induction step, it suffices to show that each vertex in $\cay(F_n; s^{\pm1})/{\lsim{\ell}}$ expands to a connected graph in $\cay(F_n; s^{\pm1}) / {\lsim{\ell+1}}$. More precisely, for all $v \in F_n$ of length~$\ell$, we will show that the subgraph of $\cay(F_n; s^{\pm1}) / {\lsim{\ell+1}}$ induced by $\{\, [x]_{\ell+1} \mid x \lsim{l} v \,\}$ is a path. 
Assume, without loss of generality, that $v$ is a reduced word, and let $a$ be its last letter (so $a \in A^{\pm1}$). We can write the cycle $\cay(F_n; s^{\pm1})/{\lsim{1}}$ as
    \[ [a^{-1}]_1 \edge [\aa_1]_1 \edge [\aa_2]_1 \edge \cdots \edge [\aa_{2r}]_1 \edge [a^{-1}]_1 , \]
where each $\aa_i$ is in $A^{\pm 1} \cup \{1\}$, and is not equal to~$a^{-1}$. For $1 \le i < 2r$, there exists $g \in F_n$, such that $[g]_1 = [\aa_i]_1$ and either $[gs]_1 = [\aa_{i+1}]_1$ or $[gs^{-1}]_1 = [\aa_{i+1}]_1$. This implies 
$[vg]_{\ell+1} = [v\aa_i]_{\ell+1}$ and either $[vgs]_{\ell+1} = [v
\aa_{i+1}]_{\ell+1}$ or $[vgs^{-1}]_{\ell+1} = [v\aa_{i+1}]_{\ell+1}$. In either case, this implies there is an edge in $\cay(F_n; s^{\pm1}) / {\lsim{\ell+1}}$ from $[v\aa_i]_{\ell+1}$ to $[v\aa_{i+1}]_{\ell+1}$. Hence, the subgraph induced by $[v]_\ell$ is the path
    \[ [v\aa_1]_{\ell+1} \edge 
    [v\aa_2]_{\ell+1} \edge 
    \cdots
    \edge [v\aa_{2r}]_{\ell+1} 
   . \qedhere \]
\end{proof}

With these results in hand, it is now easy to prove \cref{FreeUHCEg,VirtFree} of the Introduction.

\thirdmain*
\begin{proof}
The uniqueness is immediate from \cref{uniqueFn}.

By \cref{CycleToHamCircle,X1s}, the proof will be complete if we verify that the graph $X^s_1$ is a cycle in each case.

\pref{FreeUHCEg-squares} We have
    \begin{align} \tag{$*$} \label{FreeUHCEg-squaresPf-cycle}
    1 
    \stackrel{a_1}\rightarrow a_1
   \stackrel{a_1^2}\leftarrow a_1^{-1}
   \stackrel{a_1a_2}\rightarrow a_2
   \stackrel{a_2^2}\leftarrow a_2^{-1}
   \stackrel{a_2a_3}\rightarrow 
   \cdots
    \stackrel{a_{n-1}a_n}\rightarrow a_n
   \stackrel{a_n^2}\leftarrow a_n^{-1}
   \stackrel{a_n}\rightarrow 1
   . \end{align} 
    
\pref{FreeUHCEg-commutators}
We have
    \[ \raise3.1\baselineskip\hbox{$  
    \begin{matrix}
        1 & \stackrel{a_1}\rightarrow 
        & a_1
    & \stackrel{a_1^{-1} a_2^{-1}}\rightarrow & a_2^{-1}
    & \stackrel{a_2 a_1^{-1}}\rightarrow & a_1^{-1}  
    & \stackrel{a_1 a_2}\rightarrow & a_2  
    \\& \stackrel{a_2^{-1} a_3}\rightarrow 
    &a_3 
    & \stackrel{a_3^{-1} a_4^{-1}}\rightarrow & a_4^{-1}
    & \stackrel{a_4 a_3^{-1}}\rightarrow & a_3^{-1}  
    & \stackrel{a_3 a_4}\rightarrow & a_4  
    \\& \stackrel{a_4^{-1} a_5}\rightarrow 
    &a_5 
     & \stackrel{a_5^{-1} a_6^{-1}}\rightarrow & a_6^{-1}
     & \stackrel{a_6 a_5^{-1}}\rightarrow & a_5^{-1}
    & \stackrel{a_5 a_6}\rightarrow & a_6
     \\& & \vdots
     \\ & \stackrel{a_{n-2}^{-1} a_{n-1}}\rightarrow 
     & a_{n-1}
    &  \stackrel{a_{n-1}^{-1} a_{n}^{-1}}\rightarrow & a_n^{-1} 
    & \stackrel{a_n a_{n-1}^{-1}}\rightarrow & a_{n-1}^{-1}  
    & \stackrel{a_{n-1} a_n}\rightarrow & a_n  
    & \stackrel{a_n^{-1}}\rightarrow 
   &  1    
   . 
    \end{matrix}$} \qedhere \]
 \end{proof}

\begin{remark}
If a graph has a unique hamiltonian circle, then the automorphism group of the graph acts on this circle. So \cref{FreeUHCEg} provides examples of faithful actions of $F_n$ on~$S^1$ (by homeomorphisms). However, it is not clear that these examples have any special interest from a dynamical point of view. In particular, the action provided by part~\pref{FreeUHCEg-squares} of the \lcnamecref{FreeUHCEg} is unfortunately not orientation-preserving (because the arrows in equation~\pref{FreeUHCEg-squaresPf-cycle} of the proof of \fullref{FreeUHCEg}{squares} do not have a consistent orientation).
\end{remark}

Our next proof uses the following characterization of outerplanarity:

\begin{thm}[Heuer {\cite[Thm.~1.8]{Heuer}}] \label{HeuerOuterplanar}
\leavevmode
\noprelistbreak
\begin{enumerate}
    \item \label{HeuerOuterplanar-iff}
    A locally finite connected graph is outerplanar \textup(as defined in \cref{OuterPlanarDef}\textup) if and only if it is 2-connected, and has no minor that is isomorphic to either $K_4$ or $K_{2,3}$.
    \item \label{HeuerOuterplanar-UHC}
    Every outerplanar graph has a unique hamiltonian circle.
\end{enumerate}
\end{thm}

We now prove the following generalization of \cref{FnOuterplanar}.

\begin{thm} \label{MoreFnOuterplanar}
In the situation of \cref{CycleToHamCircle} \textup(which includes \fullref{FreeUHCEg}{squares} and \fullref{FreeUHCEg}{commutators}\textup), the Cayley graph
$\cay \bigl( F_n; A^{\pm1} \cup \{s^{\pm1}\} \bigr)$ is outerplanar.
\end{thm}

\begin{proof}
Let $X = \cay \bigl( F_n; A^{\pm1} \cup \{s^{\pm1}\} \bigr)$ and $C = \cay(F_n; s^{\pm1})$, so (by the proof of \cref{CycleToHamCircle}) $C/{\lsim{\ell}}$ is a hamiltonian cycle in $X/{\lsim{\ell}}$, for every~$\ell$.

We know from \cref{CycleToHamCircle} that $X$ has a hamiltonian circle. By using \cref{2factor}, it is easy to see that this implies $X$ is $2$-connected \cite[Cor.~2.9]{Heuer}.
Thus, by \fullcref{HeuerOuterplanar}{iff}, we only need to show that $X$ has no minor that is isomorphic to either $K_4$ or $K_{2,3}$. Since any finite minor of~$X$ is isomorphic to a minor of $X/{\lsim{\ell}}$ for all sufficiently large~$\ell$, it suffices to show (by induction) that $X/{\lsim{\ell}}$ is outerplanar. 

\emph{Base case.} The graph $X/{\lsim{1}}$ is obtained from the cycle $C/{\lsim{1}}$ by adding an edge joining the vertex~$[1]_1$ to each of the other vertices of the cycle. This is an outerplanar graph.

\emph{Induction step.}
Let $v$ be a word of length~$\ell$. 
When passing from $X/{\lsim{\ell}}$ to $X/{\lsim{\ell+1}}$, the vertex $[v]_\ell$ expands to the set of all vertices of the form $[va]_{\ell+1}$, for $a \in A \cup \{1\}$, such that $a$ is not the inverse of the last letter of~$v$. We know from the proof of \cref{CycleToHamCircle} that this is a set of consecutive vertices in the hamiltonian cycle~$C/{\lsim{\ell+1}}$. Therefore, the subgraph of $X/{\lsim{\ell+1}}$ that is induced by this set of vertices is a path, together with an edge from~$[v]_{\ell+1}$ to each of the other vertices of this path. Furthermore, the only other edges of $X/{\lsim{\ell+1}}$ that are incident with this subgraph are two edges of the hamiltonian cycle $C/{\lsim{\ell+1}}$ (which are incident with the endpoints of the path) and an edge from~$[v]_{\ell+1}$ to~$[w]_{\ell+1}$, where $w$ is obtained from~$v$ by removing the last letter. These edges can all be drawn inside (or on) the hamiltonian cycle (see \cref{FnOuterPlanarFig}), so $X/{\lsim{\ell+1}}$ is outerplanar.
\end{proof}

\begin{figure}[ht]
    \centering
    \begin{tikzpicture}
	       \draw (-3,0) node [circle,fill, inner sep=2pt] {};
	        \draw (-3,0) node[above] {$[v]_\ell$};
	\draw[line width=0.5mm, dashed] (-4,-0.22) arc(120:60:2cm and 1.6cm);
	\draw[line width=0.5mm, dashed] (-3,-0.22) arc(3:-90:1cm and 0.75cm);

	\draw (2,0) node[above] {$[v]_{\ell+1}$};
        \draw[line width=0.5mm] (0,0)--(5,0);
	    \foreach \i in {0,...,5}{
	       \draw (\i,0) node [circle,fill, inner sep=2pt] {};}
	\draw[line width=0.5mm] (2,0) arc(0:-180:0.5cm and 0.2cm);
	\draw[line width=0.5mm] (2,0) arc(0:-180:1cm and 0.35cm);
	\draw[line width=0.5mm] (2,0) arc(-180:0:0.5cm and 0.2cm);
	\draw[line width=0.5mm] (2,0) arc(-180:0:1cm and 0.35cm);
	\draw[line width=0.5mm] (2,0) arc(-180:0:1.5cm and 0.6cm);	
	\draw[line width=0.5mm, dashed] (0,0) arc(90:120:1.5cm and 1cm);
	\draw[line width=0.5mm, dashed] (5,0) arc(90:60:1.5cm and 1cm);
	\draw[line width=0.5mm, dashed] (2,0) arc(3:-90:2.5cm and 1cm);
    \end{tikzpicture}
    \caption{The vertex $[v]_\ell$ of $X/{\lsim{\ell}}$ expands to a set of consecutive vertices of the hamiltonian cycle in $X/{\lsim{\ell+1}}$.}
    \label{FnOuterPlanarFig}
\end{figure}
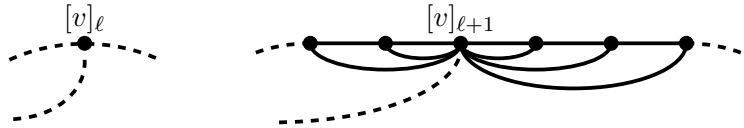




\VirtFree*

\begin{proof}
\fullCref{FreeUHCEg}{squares} provides a $(2n + 2)$-element symmetric generating set~$S_0$ of $F_n$, such that $\cay(F_n; S_0)$ has a hamiltonian circle. Let $R = \{r_1, \ldots, r_m\}$ be a set of coset representatives of $F_n \backslash G$. Then set
    \[ S \coloneqq \{\, r_{j-1}^{-1} r_j \mid j = 2,\ldots,m \,\}^{\pm1}
        \cup \{\, r_m^{-1} s r_1 \mid s \in S_0 \,\}^{\pm1} . \]
Note that $\#S \le 2(m-1) + 2(2n + 2) = 2m + 4n + 2$.

To complete the proof of the bound $2m + 4n + 2$, we show that $\cay(G; S)$ has a hamiltonian circle. Let $C$ be a hamiltonian circle in $\cay(F_n; S_0)$, so each connected component of~$C$ is a two-way-infinite path. For each component~$P$, we will construct a two-way-infinite hamiltonian path~$P'$ in the subgraph of $\cay(G; S)$ that is induced by $V(P) \cdot R$. Then it is not hard to see that the union of these new two-way-infinite paths is a hamiltonian circle in $\cay(G; S)$. 

To construct~$P'$, write $P = ( v_k \mid k \in \Z)$. Then 
                \[
                P' = \bigl( (v_k r_j : j = 1,...,m) : k \in \Z \bigr).
                \]
In other words, $P'$ is the following two-way-infinite path:
\[
\cdots \edge v_0 r_1 \edge  v_0 r_2 \edge v_0 r_3 \edge \cdots \edge v_0 r_m \edge v_1 r_1 \edge v_1 r_2 \edge \cdots
\]  

Now, suppose $F_n \trianglelefteq G$. Since $m = \#(G/F_n)$, there is a generating set $S_1$ of the finite group $G/F_n$, such that $\cay(G/F_n; S_1^{\pm1})$ has a hamiltonian path $(s_i \mid 1 \le i \le m-1)$, and $\#S_1 \le \log_2 m$ \cite[Thm.~1]{PakRadoicic}. For each $s \in S_1$, let $\hat{s}$ be a representative of~$s$ in~$G$; then let $\widehat{S_1} = \{\, \hat{s} \mid s \in S_1 \,\}$. Also let $r = \hat s_1 \hat s_2 \cdots \hat s_{m-1}$. Now, let
    \[ S' \coloneqq \widehat{S_3}^{\pm1} \cup (r^{-1} S_0)^{\pm1} . \]
Note that
\[ \#S' \le 2 \, \log_2 m + 2(2n + 2) = (2 + 2 \, \log_2 m) + 4n + 2 .\]

To complete the proof, we show that $\cay(G; S')$ has a hamiltonian circle. Note that, since $(s_i \mid 1 \le i \le m-1)$ is a hamiltonian path in $\cay(G/F_n; S_1^{\pm1})$, we know that
    \[ R \coloneqq \{\, 1, \, \hat s_1, \, \hat s_1 \hat s_2 , \, \hat s_1 \hat s_2 \hat s_3 , \ldots, \, \hat s_1 \hat s_2 \hat s_3  \cdots \hat s_{m-1} = r \,\}\]
is a set of coset representatives of $G/F_n$ in~$G$. For each connected component~$P$ of~$C$, we construct a two-way-infinite path~$P'$ exactly as above (with $r_1 = 1$ and $r_m = r$), so that the union of these new two-way-infinite paths is a hamiltonian circle in $\cay(G; S')$.
\end{proof}

It is known \cite[Cor.~22]{MiraftabMoghadamzadeh} that if a Cayley graph has a hamiltonian circle, then it admits a nowhere zero $\Z_4$-flow. Hence, the following question is related to the work in this \lcnamecref{F2Sect}.

\begin{question}
Which connected Cayley graphs of~$F_2$ (or of a more general free group) admit a nowhere zero $\Z_4$-flow? In particular, is there a connected Cayley graph of $F_2$ that does \emph{not} admit a nowhere zero $\Z_4$-flow?
\end{question}

\section{A converse for some Cayley graphs of \texorpdfstring{$F_n$}{Fn}} \label{F2Sect}

This \lcnamecref{F2Sect} proves \cref{F_2} of the Introduction:

\forthmain*

 The direction~($\Leftarrow$) is provided by \cref{FreeUHCEg}, so we will consider only~($\Rightarrow$).
 
\begin{notation}
We retain the notation of \cref{FnNotation}.
\end{notation}

\subsection{We will first prove \fullref{F_2}{degree}} \label{F2Sect-Fn}

\begin{assumption}
In this \lcnamecref{F2Sect-Fn}, we assume:
    \begin{enumerate}
    \item $\cay(F_n; s^{\pm1})$ is a hamiltonian circle in $\cay(F_n; S)$, 
    and
    \item $\#_a^\pm(s) = 2$ for all $a \in A$.
    \end{enumerate}
This implies $X_1^s$ is a cycle (cf.~\cref{DegInGr}).
\end{assumption}

For $w \in F_n$, \cref{X1s} defines $X_1^w$ to be a graph whose vertex set is $A^{\pm1} \cup \{1\}$. However, 
if $w$ happens to be in $F_k$, for some $k < n$, then the notation $X_1^w$ is ambiguous, because it does not specify whether the vertex set is $\{a_1^{\pm1},\ldots,a_k^{\pm1}, 1\}$ or $\{a_1^{\pm1}, \ldots, a_n^{\pm1}, 1 \}$. Hence, the following modification of \cref{X1s} is natural:

\begin{notation}
For a reduced word $w \in F_n$, let $A_w$ be the smallest subset of~$A$, such that $w \in \langle A_w \rangle$. (In other words, $A_w$ is the  of elements of~$A$ that appear as a letter in either~$w$ or~$w^{-1}$.) Then 
    \[ \text{$\XX_1^w$ is the subgraph of~$X_1^w$ that is induced by $A_w^{\pm1} \cup \{1\}$.} \]
\end{notation}

\begin{lem} \label{uvCycle}
Suppose $s = uv$, where $u \in \langle a_1, \ldots, a_k \rangle$ and $v \in \langle a_{k+1}, \ldots, a_n \rangle$. Then $\XX_1^u$ and $\XX_1^v$ are cycles. 
\end{lem}

\begin{proof}
Let $x_1$ and~$y_!$ be the first letters of $u$ and~$v$ (respectively, and let $x_2$ and~$y_2$ be the last letters. Also let $U$ and~$V$ be the subgraphs of~$X_1^s$ induced by $\{a_1, \ldots, a_k\}$ and $v \in \langle a_{k+1}, \ldots, a_n$, respectively. Then 
    \[ E(X_1^s) = E(U) \cup E(V) 
        \cup \{\, 1\edge x_1, \ y_1^{-1} \edge x_2, \ y_2^{-1} \edge 1 \,\} . \]
Therefore, it is clear that $X_1^s$ is a cycle if and only if $U$ and~$V$ are paths, which is the same as saying that $\XX_1^u$ and~$\XX_1^v$ are cycles.
\end{proof}

\begin{prop} \label{CommutatorConverse}
If $s \in [F_n, F_n]$, then 
    \begin{enumerate}
    \item \label{CommutatorConverse-even}
    $n$ is even, 
    and
    \item \label{CommutatorConverse-aut}
    some automorphism of~$F_n$ transforms $s$ to $[a_1, a_2] [a_3, a_4] \cdots [a_{n-1}, a_n]$.
    \end{enumerate} 
\end{prop}

\begin{proof}
\pref{CommutatorConverse-even}
Since $s \in [F_n, F_n]$ and $\#_a^\pm = 2$ for all $a \in A$, we see that $a$ and~$a^{-1}$ each appear exactly once in~$s$, for each $a \in A$. Therefore, all of the letters of~$s$ are distinct (and the length of~$s$ is exactly~$2n$), so if we write $s = s_1 s_2 \cdots s_{2n}$ (as a reduced word), then we can define a permutation $\sigma$ of $A^{\pm1} \cup \{1\}$ by
    \[ \sigma(s_i) = s_{i+1} , \]
where we use the convention that $s_0 = s_{2n+1} = 1$. Then
    \[ \sigma = (1 \ s_1 \ s_2 \ \cdots \ s_{2n})\]
is a $(2n + 1)$-cycle.

We are assuming that the graph $X_1^s$ is a cycle, so it also determines a cyclic permutation of $A^{\pm1} \cup \{1\}$. Specifically, if we call this permutation~$\chi$, then we see from \cref{X1s} that
    \[ \chi(s_i^{-1}) = s_{i+1} = \sigma(s_1) . \]
Therefore
    \[ \sigma = \chi \cdot (a_1^{\vphantom{-1}} \ a_1^{-1}) \, (a_2^{\vphantom{-1}} \ a_2^{-1}) \cdots (a_n^{\vphantom{-1}} \ a_n^{-1}) . \]
However, the two permutations $\sigma$ and~$\chi$ are $(2n + 1)$-cycles, so we know that both of them are even permutations. Therefore, the product of $2$-cycles that appears in the above formula must also be an even permutation. This means that the number of $2$-cycles in the product is even. I.e., $n$~is even.

\medbreak

\pref{CommutatorConverse-aut}
Choose some $a \in A$, such that the distance between the appearance of~$a$ in~$s$ and the appearance of~$a^{-1}$ is as small as possible. We may assume, without loss of generality, that $a$ is the first letter of~$s$. Now, let $b$ be the letter that follows~$a$. The minimality implies that the appearance of~$b^1$ is \emph{after}~$a^{-1}$, so we have $s = a b u a^{-1} w$, where $b^{-1}$ is in~$w$. Then $b \mapsto b u^{-1}$ transforms $s$ to $a b a^{-1} w'$. We may write this as $aba^{-1}w_1 b^{-1} w_2$.

If $w_1$ and $w_2$ have no letters in common, then $w_1$ has no letters in common with $b^{-1} w_2 aba^{-1}$, so we see from \cref{uvCycle} and induction that (after applying an appropriate group automorphism) each of these two words is a product of commutators of distinct letters. Then $s$ also has this form, as desired.

Thus, we may assume there is a letter~$c$ that appears in~$w_1$ and has its inverse in~$w_2$:
    \[ aba^{-1} w_1'c w_1'' b^{-1} w_2' c^{-1} w_2'' . \]
By applying the automorphism $c \mapsto (w_1')^{-1} c (w_2')^{-1}$, we may assume $w_1' = w_1'' = 1$, i.e., the word is
	\[ aba^{-1}c b^{-1} w_2' c^{-1} w_2'' . \]
Again, if $w_1$ and~$w_2$ have no letters in common (up to inverses), then induction applies. If not, then we can assume $w_2'$ is a single letter. Continuing in this way, we see that we may assume
	\[ s = a_1 a_2 a_1^{-1} a_3 a_2^{-1} a_4 a_3^{-1} \cdots a_{n-1} a_{n-2}^{-1} a_n a_{n-1}^{-1} a_n^{-1} . \]

By induction on~$n$, we will show there is an automorphism of~$F_n$ that transforms $s$ to the desired product of commutators.
Let $\varphi$ be the automorphism
	\[ a_{n-1} \mapsto a_{n-2}^{-1} a_{n-1} a_{n-2},
	\quad a_n \mapsto a_n a_{n-2}
	. \]
(All other $a_i$ are fixed.) Then
	\begin{align*} \varphi(s) 
	&= \varphi(a_1 a_2 a_1^{-1} a_3 \cdots a_{n-3} a_{n-4}^{-1} a_{n-2} a_{n-3}^{-1}  \cdot
	a_{n-1} a_{n-2}^{-1} a_n a_{n-1}^{-1} a_n^{-1})
    \\&= a_1 a_2 a_1^{-1} a_3  \cdots a_{n-3} a_{n-4}^{-1} a_{n-2} a_{n-3}^{-1}  
        \\ & \qquad \cdot
	(a_{n-2}^{-1} a_{n-1} a_{n-2}) a_{n-2}^{-1} (a_n a_{n-2}) (a_{n-2}^{-1} a_{n-1}^{-1} a_{n-2}) (a_{n-2}^{-1} a_n^{-1})
    \\&= a_1 a_2 a_1^{-1} a_3  \cdots a_{n-3} a_{n-4}^{-1} a_{n-2} a_{n-3}^{-1} \ \cdot \
	a_{n-2}^{-1} \cdot a_{n-1} a_n a_{n-1}^{-1} a_n^{-1}
	\\&= a_1 a_2 a_1^{-1} a_3 \cdots a_{n-3} a_{n-4}^{-1} a_{n-2} a_{n-3}^{-1}a_{n-2}^{-1} \cdot [a_{n-1}, a_n]	. \end{align*}
By the induction hypothesis, there is an automorphism of~$F_{n-2}$ that transforms the first part of this product to $[a_1, a_2] \cdots [a_{n-3}, a_{n-2}]$. So the entire product is transformed to the desired product of commutators.
\end{proof}

One final \lcnamecref{g'a^2}:

\begin{lem} \label{g'a^2}
If $a$, $b$, and~$c$ are distinct elements of~$A$, then there is an automorphism of the free group $\langle a,b,c \rangle$ that transforms $[a,b]c^2$ to $a^2 b^2 c^2$.
\end{lem}

\begin{proof} We have
    \begin{align*}
    [a,b]c^2
    &= a b a^{-1} b^{-1} c^2 
 	\stackrel{c \mapsto bac}\longrightarrow a b c b ac
	\stackrel{\text{conjugate by~$a^{-1}$}}\longrightarrow b c b aca
	\\&\stackrel{b \mapsto b c^{-1}}\longrightarrow b^2 c^{-1} aca
	\stackrel{a \mapsto c^{-1}a}\longrightarrow b^2 c^{-2} a^2
	\stackrel{\text{\ref{AutFn}(\ref{AutFn-perm},\ref{AutFn-sign})}}\longrightarrow a^2 b^2 c^2
	. \qedhere \end{align*}
\end{proof}

We can now complete the proof of \fullcref{F_2}{degree}.

\begin{proof}[\bf Proof of \fullcref{F_2}{degree}]
As mentioned at the start of this \lcnamecref{F2Sect}, the direction~($\Leftarrow$) was established in \cref{FreeUHCEg}, so we will only prove~($\Rightarrow$). Also, we may assume $s \notin [F_n, F_n]$, for otherwise \cref{CommutatorConverse} applies. 

We claim that some automorphism of~$F_n$ transforms $s$ to a word of the form $g a_{k+1}^2 a_{k+2}^2 \cdots a_n^2$, where $g \in [F_k, F_k]$. The proof is by induction. Since $s \notin [F_n, F_n]$, there is some element of~$A$ whose appearances in~$s$ both have the same exponent; assume, without loss of generality, that both appearances of~$a_n$ have exponent~$1$, and that $a_n$ is the last letter of~$s$, so we may write
    \[ s = w_1 a_n w_2 a_n . \]
Then the automorphism $a_n \mapsto w_2^{-1} a_n$ transforms~$s$ to $w_1 w_2^{-1} a_n^2 = w a_n^2$. By \cref{uvCycle}, we know that $\XX_1^w$ is a cycle. Now:
    \begin{itemize}
    \item If $w \in [F_{n-1},F_{n-1}]$, then the proof of the claim is complete.
    \item If not, then the induction hypothesis provides an automorphism of~$F_{n-1}$ that transforms~$w$ to $g a_{k+1}^2 a_{k+2}^2 \cdots a_{n-1}^2$, where $g \in [F_k, F_k]$. This  automorphism transforms~$s$ to precisely the desired form.
    \end{itemize}
This completes the proof of the claim.

\medbreak

From \cref{uvCycle}, we know that $\XX_1^g$ is a cycle.
Therefore, \cref{CommutatorConverse} provides an automorphism of~$F_k$ that transforms 
    \[ \text{$g a_{k+1}^2 a_{k+2}^2 \cdots a_n^2$ to $[a_1,a_2] [a_3,a_4] \cdots [a_{k-1},a_k] a_{k+1}^2 a_{k+2}^2 \cdots a_n^2$.} \]
Repeated application of \cref{g'a^2} transforms this to $a_1^2 a_2^2 \cdots a_n^2$.
\end{proof}

\subsection{Next, we will prove \fullref{F_2}{n=2}}

\begin{notation}
We will use $a$ and~$b$ for the standard generators of~$F_2$, instead of $a_1$ and~$a_2$.
\end{notation}

\begin{definition}[cf.~{\cite{WhiteheadAlg}}]
Let us say that an element $s$ of~$F_n$ is \emph{automorphically minimal} if $\ell(s) \leq  \ell \bigl( \varphi (s) \bigr)$, for all $\varphi \in \autt(F_n)$.
\end{definition}

\begin{lem} \label{BlocksAreConnected}
Let $s = s_1 s_2 \cdots s_n$ be a nontrivial reduced word in $F_2$, such that $s$ is automorphically minimal and $s_1 \neq s_n$.
Also let $\ell \ge 2$ and $v \in F_2$.
Then the subgraph of $\cay(F_2; s^{\pm1})/{\lsim{\ell}}$ that is induced by $\{\, [x]_\ell \mid x \lsim{\ell-1} v \,\}$ is connected.
\end{lem}

\begin{proof}
We may assume that $v$ is a reduced word of length~$\le \ell - 1$. In fact, we may assume the length of~$v$ is exactly~$\ell - 1$, because otherwise the set has only one vertex, so obviously induces a connected subgraph. Furthermore, we may assume, without loss of generality, that the last letter of~$v$ is~$b$, so $v = wb$, where $w$ is a reduced word of length~$\ell - 2$ (and $w$ does not end with~$b^{-1}$). 

We wish to show that the subgraph of $\cay(F_2; s)/{\lsim{\ell}}$ induced by 
    \[ \bigl\{ [wb]_\ell, [wbb]_\ell, [wba]_\ell, [wba^{-1}]_\ell \bigr\}\] 
is connected.
Note that $s_1 \neq s_n^{-1}$ (because the automorphic minimality of~$s$ implies that the conjugate $s_1^{-1} s s_1$ cannot have strictly shorter length than~$s$). Since we also assume $s_1 \neq s_n$, this implies $\{s_1^{\pm1}, s_n^{\pm1}\} = \{a^{\pm1}, b^{\pm1}\}$. Therefore, we may assume without loss of generality that 
    \[ \text{the first letter of~$s$ is~$a^{-1}$  and the last letter is either $b$ or~$b^{-1}$} \]
(perhaps after replacing $s$ with its inverse and/or interchanging $a$ with~$a^{-1}$), so 
	\[ [wb]_\ell \edge [wba^{-1}]_\ell .\]
(Writing ``$[x]_\ell \edge [y]_\ell$'' means $[x]_\ell$ is adjacent to $[y]_\ell$ in $\cay(F_2; s^{\pm1})/{\lsim{\ell}}$.)
 
\refstepcounter{caseholder} 

\begin{case}
Assume either $a a$ or its inverse appears \textup(as two consecutive letters\textup) in~$s$.
\end{case}
This implies
	\[ [wba^{-1}]_\ell \edge [wba]_\ell .\]
Now, if either $ab$ appears or $a^{-1}b$ appears, then we have either $[wba^{-1}]_\ell \edge [wbb]_\ell$ or $[wba]_\ell \edge[wbb]_\ell$. In either case, we see that the subgraph is connected, as desired.

So we may now assume that neither $ab$ nor $a^{-1}b$ appears. Therefore, there is no occurrence of~$b$ to a positive power in the word~$s$. However, we also know that $s_n \in \{b^{\pm1}$. So the last letter of`$s$ must be~$b^{-1}$. Then the first letter of~$s^{-1}$ is~$b$, so $[wb]_\ell \edge [wbb]_\ell$. Hence, the subgraph is connected.

\begin{case} \label{BlocksAreConnectedPf-NoAA}
Assume that neither $a a$ nor its inverse appears in~$s$.
\end{case}

We claim that 
	\[ [wbb]_\ell  \edge [wba]_\ell . \]
Suppose not. Then $b^{-1} a$ does not appear anywhere (in either $s$ or~$s^{-1}$). Since the first letter of~$s$ is not~$a$, this implies that every occurrence of~$a$ is preceded by~$b^{-1}$ (and every occurrence of~$a^{-1}$ is followed by~$b$), which contradicts the minimality of the length of~$s$ (because there is an automorphism~$\varphi$ of~$F_2$, such that $\varphi(a) = ba$ and $\varphi(b) = b$). 

Now:
\noprelistbreak 
	\begin{itemize}
	\item If $ab$ appears, then $[wba^{-1}]_\ell \edge [wbb]_\ell $, so the subgraph is connected, as desired.
	\item If the last letter of~$s$ is~$b^{-1}$, then the first letter of~$s^{-1}$ is~$b$, so $[wb]_\ell \edge [wbb]_\ell $, so the subgraph is connected, as desired.
	\end{itemize}
Therefore, we may assume that $ab$ does not appear in either $s$ or~$s^{-1}$, and that the last letter of~$s$ is~$b$.

We know that $ab$ does not appear in either $s$ or~$s^{-1}$, and (from the assumption of this \lcnamecref{BlocksAreConnectedPf-NoAA}) that $aa$ also does not appear. Since the last letter of~$s$ is~$b$, this implies that the same is true with $b s b^{-1}$ in the place of~$s$. Then every occurrence of~$a$ in $b s b^{-1}$ must be followed by~$b^{-1}$ (and every occurrence of~$a^{-1}$ must be preceded by~$b$). This contradicts the minimality of the length of~$s$ (because there is an automorphism~$\varphi$ of~$F_2$, such that $\varphi(a) = ab$ and $\varphi(b) = b$).
\end{proof}

\begin{proof}[\bf Proof of \fullcref{F_2}{n=2}]
Like before, we will only prove~($\Rightarrow$), because the direction~($\Leftarrow$) was established in \cref{FreeUHCEg}.
It suffices to show that $\cay(F_2; s^{\pm1})/{\lsim{1}}$ is a cycle \textup(with no multiple edges\textup), for then we see from \cref{DegInGr} that $\#_a^\pm(s) = \#_b^\pm(s) = 2$, so \fullref{F_2}{degree} applies.

For convenience, let $X_\ell^s = \cay(F_2; s^{\pm1})/{\lsim{\ell}}$ (for each $\ell \ge 1$). 
We see from \fullcref{2factor}{even} (and \cref{FnCut}) that the degree of each vertex of~$X_\ell^s$ is even. (So $X_\ell^s$ has an Euler tour.)
If $X_1^s$ is not a cycle (or has a multiple edge), this implies it has two (non-loop) edges $e_1$ and~$e_2$, such that removing these two edges from $X_1^s$ results in a connected graph. Let $\widetilde{e_1}$ and~$\widetilde{e_2}$ be the representatives of $e_1$ and~$e_2$ in $\cay(F_n; s^{\pm1})$. For each~$\ell$, let $X_\ell'$ be the graph that is obtained from~$X_\ell^s$ by removing $\widetilde{e_1}$ and~$\widetilde{e_2}$.

By the choice of $e_1$ and~$e_2$, we know that $X_1'$ is connected. Also, we see from \cref{BlocksAreConnected} that each vertex of~$X_{\ell-1}'$ expands to a connected subgraph of~$X_\ell'$. By induction on~$\ell$, this implies that each $X_\ell'$ is connected. Hence, every edge cut of $\cay \bigl( F_2; A^{\pm1} \cup \{s^{\pm1}\} \bigr)/{\lsim{\ell}}$ must contain some edge of~$X_\ell^s$ that is not in $\{e_1,e_2\}$. This contradicts \fullcref{2factor}{2edges} (combined with \cref{FnCut}).
\end{proof}

\begin{acknowledgments}
We thank Sean Legge for sharing \cref{LeggeEg} and allowing us to include it in our paper.
We also thank Agelos Georgakopoulos for several useful discussions and for providing \cref{uniqueFn}.

This research was conducted while the first author had a postdoctoral position at the Department of Mathematics and Computer Science of the University of Lethbridge.

Our main results on free groups were suggested by the output of 
\textsf{sagemath} \cite{sagemath} computer programs. These programs are available online at
\raggedright
\url{https://arxiv.org/src/2303.11124v2/anc/}.
\end{acknowledgments}

\bibliographystyle{Miraftab-Morris}
\raggedright
\bibliography{references.bib}
\end{document}